\documentclass[11pt]{article}
\usepackage{amsfonts,amsthm, amsmath}
\usepackage{epsfig}
\usepackage{makeidx}
\usepackage{graphicx,epstopdf}
\usepackage{float}
\usepackage{amsmath}
\usepackage{comment}
\usepackage{enumerate}
\usepackage{subfig}
\usepackage{hyphenat}
\DeclareGraphicsExtensions{.ps,.png,.jpg}
\usepackage{cite}

\makeindex
\usepackage{color}
\usepackage{hyperref}

\newcommand{\ncom}{\newcommand}
\ncom{\ul}{\underline}
\ncom{\beq}{\begin{equation}}
\ncom{\eeq}{\end{equation}}
\ncom{\bea}{\begin{eqnarray*}}
\ncom{\eea}{\end{eqnarray*}}
\ncom{\beqa}{\begin{eqnarray}}
\ncom{\eeqa}{\end{eqnarray}}
\ncom{\nno}{\nonumber}
\ncom{\non}{\nonumber}
\ncom{\ds}{\displaystyle}
\ncom{\half}{\frac{1}{2}}
\ncom{\mbx}{\makebox{.25cm}}
\ncom{\hs}{\mbox{\hspace{.25cm}}}
\ncom{\rar}{\rightarrow}
\ncom{\Rar}{\Rightarrow}
\ncom{\noin}{\noindent}
\ncom{\bc}{\begin{center}}
\ncom{\ec}{\end{center}}
\ncom{\sz}{\scriptsize}
\ncom{\rf}{\ref}
\ncom{\s}{\sqrt{2}}
\ncom{\sgm}{\sigma}
\ncom{\Sgm}{\Sigma}
\ncom{\psgm}{\sigma^{\prime}}
\ncom{\dt}{\delta}
\ncom{\Dt}{\Delta}
\ncom{\lmd}{\lambda}
\ncom{\Lmd}{\Lambda}
\ncom{\Th}{\Theta}
\ncom{\e}{\eta}
\ncom{\eps}{\epsilon}
\ncom{\pcc}{\stackrel{P}{>}}
\ncom{\lp}{\stackrel{L_{p}}{>}}
\ncom{\dist}{{\rm\,dist}}
\ncom{\sspan}{{\rm\,span}}
\ncom{\re}{{\rm Re\,}}
\ncom{\im}{{\rm Im\,}}
\ncom{\sgn}{{\rm sgn\,}}
\ncom{\ba}{\begin{array}}
\ncom{\ea}{\end{array}}
\ncom{\hone}{\mbox{\hspace{1em}}}
\ncom{\htwo}{\mbox{\hspace{2em}}}
\ncom{\hthree}{\mbox{\hspace{3em}}}
\ncom{\hfour}{\mbox{\hspace{4em}}}
\ncom{\vone}{\vskip 2ex}
\ncom{\vtwo}{\vskip 4ex}
\ncom{\vonee}{\vskip 1.5ex}
\ncom{\vthree}{\vskip 6ex}
\ncom{\vfour}{\vspace*{8ex}}
\ncom{\norm}{\|\;\;\|}
\ncom{\integ}[4]{\int_{#1}^{#2}\,{#3}\,d{#4}}
\ncom{\vspan}[1]{{{\rm\,span}\{ #1 \}}}
\ncom{\dm}[1]{ {\displaystyle{#1} } }
\ncom{\ri}[1]{{#1} \index{#1}}

\newtheorem{theorem}{\bf Theorem}[section]
\newtheorem{remark}{\bf Remark}[section]
\newtheorem{proposition}{Proposition}[section]

\newtheoremstyle
    {remarkstyle}
    {}
    {11pt}
    {}
    {}
    {\bfseries}
    {:}
    {     }
    {\thmname{#1} \thmnumber{#2} }

\theoremstyle{remarkstyle}

\begin{document}

\begin{center}
{\Large \bf GARTFIMA  Process and its Empirical Spectral Density Based Estimation}
\end{center}
\vone
\vone
\begin{center}
{Niharika Bhootna}$^{\textrm{a}}$, {Arun Kumar}$^{\textrm{a*}}$
 		$$\begin{tabular}{l}
 		$^{\textrm{a}}$ \emph{Department of Mathematics, Indian Institute of Technology Ropar, Rupnagar, Punjab - 140001, India}\\
 		$^{\textrm{*}}$ \emph{Correspondence: arun.kumar@iitrpr.ac.in}
 \end{tabular}$$
 \end{center}
\vtwo
\begin{center}
\noindent{\bf Abstract}
\end{center}
In this article, we introduce a Gegenbauer autoregressive tempered fractionally integrated moving average (GARTFIMA) process. We work on the spectral density and autocovariance function for the introduced process. The parameter estimation is done using the  empirical spectral density with the help of the nonlinear least square technique and  the Whittle likelihood estimation technique. The performance of the proposed estimation techniques is assessed on simulated data. Further, the introduced process is shown to better model the real-world data in comparison to other time series models.

\vtwo

\vone

\section{Introduction}
The autoregressive (AR), moving average (MA) and autoregressive moving average (ARMA) models were introduced by Yule (1926) \cite{Yule1926}, Slutsky (1937) \cite{Slutsky1937} and Wold (1938) \cite{Wold1938} respectively to model stationary time series data. A time series is stationary if its statistical properties remain unchanged over time. Further, it is argued in \cite{Wold1938} that all stationary time series can be modeled using ARMA processes as long as the appropriate number of AR and MA terms are specified. However, in many real-life scenarios the observed data need not be stationary. Latter in 1976 Box $\&$ Jenkins \cite{Box1976} developed the autoregressive integrated moving average ARIMA$(p, d, q)$ process to model non-stationary time series by introducing the difference operator $d$ in the traditional ARMA$(p, q)$ process. The ARIMA$(p, d, q)$ process is concerned with modeling a non-stationary time series by taking $d$th order differencing of the process to make it stationary and then using ARMA$(p, q)$ model with appropriate $p$ and $q$. However, data exhibiting long-range dependence (LRD) cannot be modeled using ARIMA$(p, d, q)$ process with $d$ as an integer. The data exhibiting LRD behaviours have been found in various fields such as finance, economics, geophysics and agriculture see e.g. \cite{Beran1994, Robinson2003} and references therein. The LRD series evince substantial correlation after large lags. The LRD process is essential to study as they exhibit non-instinctive properties which may not be captured using the traditional time series models. Hosking (1981) \cite{Hosking1981} generalized the Box $\&$ Jenkins approach by introducing a fractional differencing operator which can be represented by infinite binomial series expansion in powers of backward shift operator in the ARMA$(p, d, q)$ model. The process is known as autoregressive fractionally integrated moving average or in short called ARFIMA$(p, d, q)$ process. Here $d\in\mathbb{R}$ and the process is stationary for $|d|<1/2$ and it has LRD property for $d\in(0, 0.5)$. For a long memory stochastic process ${X_t}$, the autocorrelation function for a lag $h$ denoted by $\rho(h)$ behaves asymptotically as $\rho(h) \sim K h^{-\alpha}$ as $h\rightarrow  \infty$ for some nonzero constant $K$ and positive $\alpha \in (0, 1)$, see \cite{Chung1996, Beran1994}. Alternatively, an LRD process is a process whose autocovariance function or autocorrelation function is not absolutely summable or in the frequency domain, it is a process whose power spectral density is not everywhere bounded. In recent years tempered processes are studied in great detail \cite{Rosinski2007, Kumar2014, Sabzikar2018, Gupta2020, Bhootna2022}. These processes are obtained by exponential tempering in the original process.  The ARTFIMA models were introduced as a generalization of ARFIMA model in \cite{Sabzikar2019}. According to \cite{Sabzikar2019}, it is more convenient to study ARTFIMA time series as the covariance function is absolutely summable in finite variance cases and the spectral density converges to zero \cite{Gray1989}. Also in the spectral density of ARFIMA process is unbounded as the frequency approaches to $0$. In many scenarios, there exist many datasets for which the power spectrum is bounded as frequency approaches to $0$ and these types of datasets cannot be models using ARFIMA process. An extension of the Fractional ARIMA process is proposed to model long-term seasonal and periodic behaviours and is referred to as Gegenbauer autoregressive moving average (GARMA) process \cite{Gray1989}. The GARMA process uses the properties of Gegenbauer generating function to model a time series. For Gegenbaur random fields see e.g. \cite{Espejo2014}.\\

\noindent In this article we introduce and study Gegenbauer autoregressive integrated moving average (GARTFIMA) process, which is a further extension to GARMA process as well as ARTFIMA process by including the tempering parameter $\lambda$ and Gegenbauer shift operator in ARFIMA process. The GARTFIMA process can be used to model datasets with periodic behaviours having a bounded spectrum for frequencies near $0$. 

The rest of the paper is organized as follows. Section $\ref{section2}$ provides the study of some predefined fractionally integrated processes named ARFIMA, ARTFIMA and GARMA processes. Also, the section includes the basic properties of these processes such as stationarity, long memory, the form of spectral density and autocovariance function. Section $\ref{section2}$ also includes a new parameter estimation technique for the ARTFIMA process which is based on non-linear regression introduced by Reisen (1994) \cite{Reisen1994}. A simulation study is also done to check the performance of the estimation method. In section $\ref{section3}$, the GARTFIMA process is introduced and the corresponding form of spectral density is obtained. Further, the autocovariance of the process is found by taking the inverse Fourier transform of the spectral density which does not have an explicit closed form expression and can be represented using the coefficients of Gegenbauer polynomials. Moreover, the stationarity and invertibility conditions are also discussed for GARTFIMA process. The parameter estimation techniques for GARTFIMA process are provided in section \ref{section4}. To estimate the parameters of the process first approach is based on non-linear regression and the second is based on the Whittle likelihood estimation and the performance is assessed by doing a simulation study for both the methods where the comparison of actual and estimated parameters are demonstrated using the box-plots. The comparison of the defined model with ARFIMA and ARTFIMA models are also presented in this section. The last section concludes the paper.

\section{Some fractionally integrated processes}\label{section2}
In this section, we recall the main properties of ARMA, ARIMA, ARFIMA, ARTFIMA and GARMA processes, which are used in Section 3 to introduce and study the GARTFIMA process.
The well-known ARMA$(p, q)$ process, which incorporates both the autoregressive AR and moving average MA models with lags $p$ and $q$ respectively, is the most effectual way of modeling a stationary time series. The process is defined as follows
\begin{align}\label{def_arma}
\Phi(B)X_t=\Theta(B) \epsilon_t,
\end{align}
where $\epsilon_t$ is Gaussian white noise with variance $\sigma^2$. Also, $\Phi(B)$ and $\Theta(B)$ are ARMA operators defined as
\begin{equation}\label{ARMA_operators}
\Phi(B)=1-\sum_{j=1}^{p} \phi_j B^j \;\; \mbox{and}\;\; \Theta(B)=1+\sum_{j=1}^{q}\theta_j B^j,
\end{equation}
where $B$ is the backward shift operator defined as $B^j(X_t)=X_{t-j}$. 
The ARMA process does not work with series exhibiting non-stationary behavior. As a result, in many scenarios, ARMA processes are unsuitable for modeling occurrences. By introducing the difference parameter $d$, the autoregressive integrated moving average (ARIMA) process is an extension of the ARMA process described above. The goal is to use integer order differencing to make data stationary. Using, integer order differencing operator $(1-B)^d$ in \eqref{def_arma}, the ARIMA process is defined by
\begin{align}\label{def_ARIMA}
\Phi(B) (1-B)^d X_t= \Theta(B) \epsilon_t,
\end{align}
where $d\in \{0,1,2,\hdots\}$ is the order of differencing. The models like AR, MA, and ARMA are special cases of this general ARIMA$(p, d, q)$ model defined in \eqref{def_ARIMA}. In practice, for AR, MA, ARMA and ARIMA models, the unknown parameters can be estimated using the least square estimation or the maximum likelihood estimation easily using a direct library available in different programming languages. 

\noindent The next subsection covers some fractionally integrated processes which are generalizations of ARMA process which are useful in modeling of non-stationary and long-memory type real-world data.

\subsection{Autoregressive fractionally integrated moving average (ARFIMA) Process}\label{ARFIMA}
The autocorrelation function for classical time series models such as AR, MA, ARMA, and ARIMA processes decays exponentially, making them short memory processes. However, this characteristic does not appear to be common in many empirical time series. Some series appears to have a long memory property, which means that the correlation between data in a series can persist over a longer time periods. The traditional models cannot be used in such scenarios. The ARFIMA process is a generalized ARIMA$(p, d, q)$ model defined  in \cite{Hosking1981}, which aims to explain both short-term and long-term persistence in data. The model is generalized by taking the order of differencing $d$ to be any real value instead of being an integer in the traditional ARIMA process. For some recent articles with applications of ARFIMA process in modelling of series with long memory see e.g. \cite{Vera-Valdes2021, Lahmiri2021, Katris2021, Mainassara2021, Peters2021}. The fractional ARIMA$(p, d, q)$ or ARFIMA$(p, d, q)$ process is defined as follows \cite{Hosking1981}
\begin{align}\label{arfima}
\Phi(B) (1-B)^d X_t= \Theta(B) \epsilon_t,
\end{align}
where $(1-B)^d$ is defined as follows 
\begin{align*}
    (1-B)^d=\sum_{k=0}^{\infty}\frac{\Gamma(k-d)B^k}{\Gamma(-d)\Gamma(k+1)}.
\end{align*}

\noindent The process is stationary and invertible for $|d|\leq \frac{1}{2}$. Rewrite \eqref{arfima} as follows 
\begin{align*}
    X_t=\Psi_x(B)\epsilon_t,
\end{align*}
where $\Psi(B)= \dfrac{\Theta(B)}{\Phi(B)} \Delta^{d}$ and $\Delta^{d}=(1-B)^{-d}$. Then for $z=e^{-\iota\omega}$ the spectral density of $X_t$ takes the following form \cite{Hosking1981}
\begin{align*}
f_x(\omega)&=|\Psi(z)|^2f_\epsilon(\omega)= \frac{\sigma^2}{2\pi}\frac{|\Theta(z)|^2}{|\Phi(z)|^2} (2\sin(\omega/2))^{-2d}.
\end{align*}

\noindent For $p=0$ and $q=0$ the spectral density takes the following form
\begin{align*}
f_x(\omega)=(2\sin(\omega/2))^{-2d}
\end{align*}

\noindent and the corresponding autocorrelation function $\gamma(h)$ for lag $h$ can be calculated by taking the inverse Fourier transform of spectral density. The relation between autocovariance function and spectral density is given by
\begin{align}\label{1}
\gamma(h)&= \int_{-\pi}^{\pi}f_x(\omega) e^{\iota\omega h} d\omega.
\end{align}
\noindent For ARFIMA process with lags $p=0$ and $q=0$ the autocovariance is given by \cite{Hosking1981}
\begin{align*}
\gamma(h)&=\frac{(-1)^h(-2d)!}{(h-d)!(-h-d)!}
\end{align*}

\subsection{Autoregressive tempered fractionally integrated moving average (ARTFIMA) process} \label{ARTFIMA} 

In this section, we introduce the main properties of ARTFIMA process. Also, we provide an alternative technique to estimate the parameters of the ARTFIMA process. The efficacy of the estimation procedure is checked on the simulated data. The ARTFIMA$(p, d, \lambda, q)$ process defined by \cite{Sabzikar2019} generalizes the ARFIMA process defined in \eqref{arfima}. The model is defined by introducing a tempering parameter $\lambda$ in ARFIMA model, that is instead of taking the fractional shift operator $(1-B)^d$ the authors have considered a tempered fractional shift operator $(1-e^{-\lambda}B)^d$, where $\lambda>0$. The series exhibits semi-long range dependence structure that is for $\lambda$ close to $0$ the autocovariance function of the process behaves similar to a long memory process and for large $\lambda$ the autocovariance  function decays exponentially \cite{Sabzikar2019}. The ARTFIMA process is defined by
\begin{align}\label{artfima}
\Phi(B) (1-e^{-\lambda}B)^d X_t= \Theta(B) \epsilon_t.
\end{align}
The equation \eqref{artfima} is rewritten as
\begin{align*}
      X_t=\Psi_x(B)\epsilon_t,
\end{align*}

\noindent where $\Psi(B)= \dfrac{\Theta(B)}{\Phi(B)} \Delta^{d,\lambda}$ and $\Delta^{d,\lambda}=(1-e^{-\lambda}B)^{-d}$. The spectral density of $X_t$ takes the following form \cite{Sabzikar2019}
\begin{align}\label{2}
    f_x(\omega)=|\Psi(z)|^2f_\epsilon(\omega)
    =\frac{|\Theta(e^{-\iota\omega})|^2}{|\Phi(e^{-\iota\omega})|^2} |1-e^{-(\lambda+\iota\omega)}|^{-2d}
    =\frac{|\Theta(e^{-\iota\omega})|^2}{|\Phi(e^{-\iota\omega})|^2}(1-2e^{-\lambda}\cos(\omega)+e^{-2\lambda})^{-d}.
\end{align}

\noindent Taking $Y_t=\Delta^{d,\lambda}X_t$ the model takes the following form
\begin{align*}
Y_t-\sum_{j=1}^{p}\phi_jY_{t-j}=\epsilon_t+\sum_{i=1}^{q}\theta_i\epsilon_{t-i},
\end{align*}
which is the ARMA$(p, q)$ process. The ARTFIMA process is stationary for $d\in \mathbb{Z}$ and $\lambda>0$ (see \cite{Sabzikar2019}). For $p=0$ and $q=0$ the autocovariance function of ARTFIMA$(0, d, \lambda, 0)$ process is defined as

\begin{align*}
\gamma(h)=\frac{\sigma^2\Gamma(2d+h)}{\Gamma(2d)\Gamma(h+1)} {}_2F_1(2d,h+2d;h+1;e^{-2\lambda}),
\end{align*}
where ${}_2F_1(a,b;c;z)$ is Gaussian hypergeometric function defined as

\begin{align*}
{}_2F_1(a,b;c;z)=1+\frac{a.b}{c.1}z+\frac{a(a+1)b(b+1)}{c(c+1)2!}z^2+\hdots.
\end{align*}

\subsubsection{Parameter estimation for ARTFIMA process}
The parameter estimation for ARTFIMA process is discussed in \cite{Sabzikar2019} using the Whittle likelihood based estimation technique. Here we provide an alternative estimation technique based on empirical spectral density using non-linear least square approach to estimate the parameters $d$ and $\lambda$. We establish a nonlinear regression equation between empirical and actual spectral densities to estimate the unknown parameters.
\noindent In \eqref{artfima}, assuming $(1-e^{-\lambda}B)X_t=U_t$, the spectral density for the ARMA$(p, q)$ process $\phi(B)U_t=\theta(B)\epsilon_t$, is given by
\begin{align}\label{SPD1}
f_u(\omega)=\dfrac{\sigma^2}{2\pi}\dfrac{|\Theta(e^{-\iota\omega})|^2}{|\Phi(e^{-\iota\omega})|^2}.
\end{align}
Using \eqref{SPD1} and \eqref{2}, the spectral density of $X_t$ can be written as

\begin{align}\label{S1}
f_x(\omega)&=f_u(\omega)(1-2e^{-\lambda}\cos(\omega)+e^{-2\lambda})^{-d}.
\end{align}

\noindent Consider $\omega_j=\dfrac{2\pi j}{n}$, $j=0,1,\cdots,[n/2]$. Here $n$ is the sample size and $\omega_j$ is the set of harmonic frequencies. Taking natural logarithm on both sides of \eqref{S1} and with some manipulation, it follows
\begin{align*}
 \log\{f_x(\omega_j)\}&=\log\{f_u(0)\}+\log\{f_u(\omega_j)\}-d \log\{(1-2e^{-\lambda}\cos(\omega)+e^{-2\lambda})\}-\log\{f_u(0)\}\\
 &=\log\{f_u(0)\}-d \log\{(1-2e^{-\lambda}\cos(\omega)+e^{-2\lambda})\}+\log\Bigg\{\dfrac{f_u(\omega_j)}{f_u(0)}\Bigg\}.
\end{align*}
Adding $\log\{I(\omega_j)\}$ on both sides
\begin{align}\label{ref1}
 \log\{I(\omega_j)\}&=\log\{f_u(0)\}-d \log\{(1-2e^{-\lambda}\cos(\omega)+e^{-2\lambda})\}+\log\Bigg\{\dfrac{f_u(\omega_j)}{f_u(0)}\Bigg\}+\log\Bigg\{\dfrac{I(\omega_j)}{f(\omega_j)}\Bigg\},
\end{align}
where $I(\omega_j)$ is known as the periodogram or empirical spectral density of the process $\{X_t\}$ stated as 
\begin{align}\label{Iw}
I_x(\omega)=\dfrac{1}{2\pi}\Bigg\{R(0)+\sum_{s=1}^{n-1}R(s)\cos(s\omega)\Bigg\}\hspace{4mm}\omega\in [-\pi,\pi],
\end{align}
where $R(s)=\dfrac{1}{n}\sum_{i=1}^{n-s}(X_i-\Bar{X})(X_{i+s}-\Bar{X}), \; s=0,1,\hdots,(n-1)$ is the sample autocovariance function with sample mean $\bar{X}$. Using \eqref{SPD1}, we have
\begin{align*}
    f_u(0)=\frac{\sigma^2}{2\pi}\frac{(1-\theta_1-\hdots-\theta_q)^2}{(1-\phi_1-\hdots-\phi_p)^2}.
\end{align*}

\noindent Assuming $\log(I(\omega_j)/f(\omega_j))$ to be the error term for the non-linear equation given in \eqref{ref1} and minimizing the sum of squared errors for the equation, the parameters $d$ and $\lambda$ are estimated. Also choosing the $\omega_j$ near $0$, the term $\log(f_u(\omega_j)/f_u(0))$ will be negligible. So another way for the estimation is to choose the upper limit of $j$ such that the $\omega_j$ is small or near zero. Considering $\omega_j$ close to 0, the equation \eqref{ref1} can be rewritten as

\begin{align}\label{ref2}
 \log\{I(\omega_j)\}&=\log\{f_u(0)\}-d \log\{(1-2e^{-\lambda}\cos(\omega)+e^{-2\lambda})\}+\log\Bigg\{\dfrac{I(\omega_j)}{f(\omega_j)}\Bigg\}.
\end{align}

\noindent The above equation can be expressed in form of a nonlinear regression equation where $\log(I(\omega_j)/f(\omega_j))$ can be expressed as error term $\log\{f_u(0)\}$ can be expressed as intercept and the parameters $d$ and $\lambda$ can be estimated by applying the nonlinear least square regression or minimizing the sum of squared errors using generalized simulated annealing with ``GenSa'' package in R.

\subsubsection{Simulation study}
A simulation study is carried out to assess the performance of the stated estimation approach. Assuming the innovation distribution to be Gaussian with mean $\mu=0$ and variance $\sigma^2=2$, a sample with size 1000 is simulated for an ARTFIMA$(p, d, \lambda, q)$ process. The simulations are carried out using the artfima package available in R. We run two distinct simulations using different parameter combinations to test the performance of the specified estimation approach. The results are summarized in table \ref{table:1}.

\begin{table}[ht!]
\begin{center}
\begin{tabular}{||c||c||c||c||}
\hline 
& Actual  & Estimated  \\ 
\hline 
Case 1  & $d$ = 0.4, $\lambda$ = 0.2 & $\hat{d}$ = 0.41, $\hat{\lambda}$ = 0.18 \\ 
\hline 
Case 2  & $d$ = 0.5, $\lambda$ = 0.4 & $\hat{d}$ = 0.52, $\hat{\lambda}$ = 0.39 \\
\hline
\end{tabular}
\caption{Estimated values for differencing parameter $d$ and tempering parameter $\lambda$.}\label{table:1}
\end{center}
\end{table}
 
 \noindent Furthermore, using the ARTFIMA$(1, 0.4, 0.2, 1)$ process, we simulate 1000 samples each with size 1000 and construct box-plots of estimated parameters using simulated data. The plot is shown in left panel of Fig \ref{fig:1}. From Fig \ref{fig:1}, it is evident that the median of the estimated parameters is equal to the true values with some outliers also. Similarly, using the ARTFIMA$(1, 0.5, 0.4, 1)$ an another simulation is performed and the box-plot for the same is given in right panel of Fig \ref{fig:1}. The performance of the estimation method also depends on the optimization method used in minimizing the errors in the non-linear optimization and the initial guess used for the parameters. 

\begin{figure}%
    \centering
    \subfloat[\label{fig:1a}]{{\includegraphics[width=8cm]{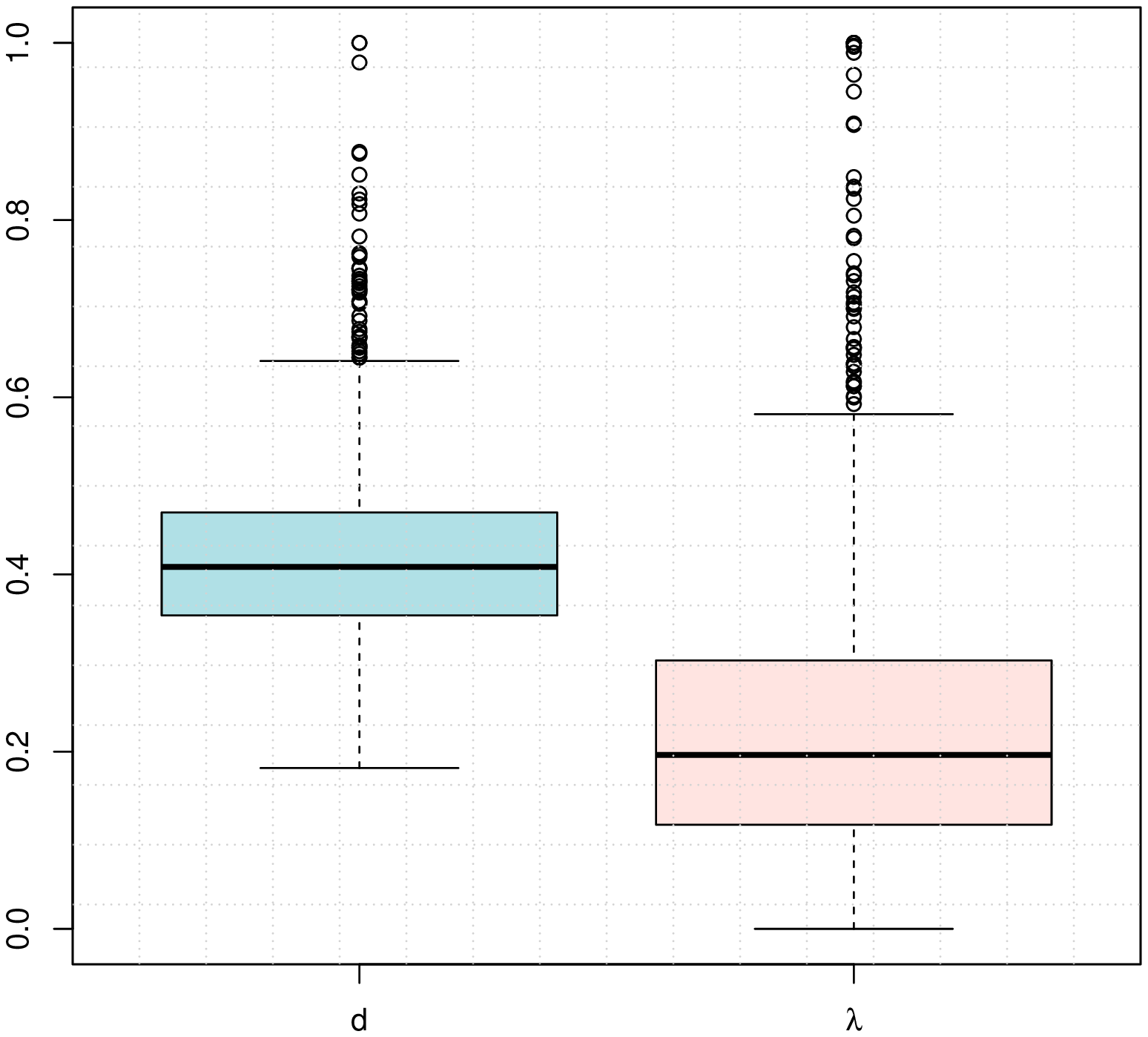} }}%
    \qquad
    \subfloat[\label{fig:1b}]{{\includegraphics[width=8cm]{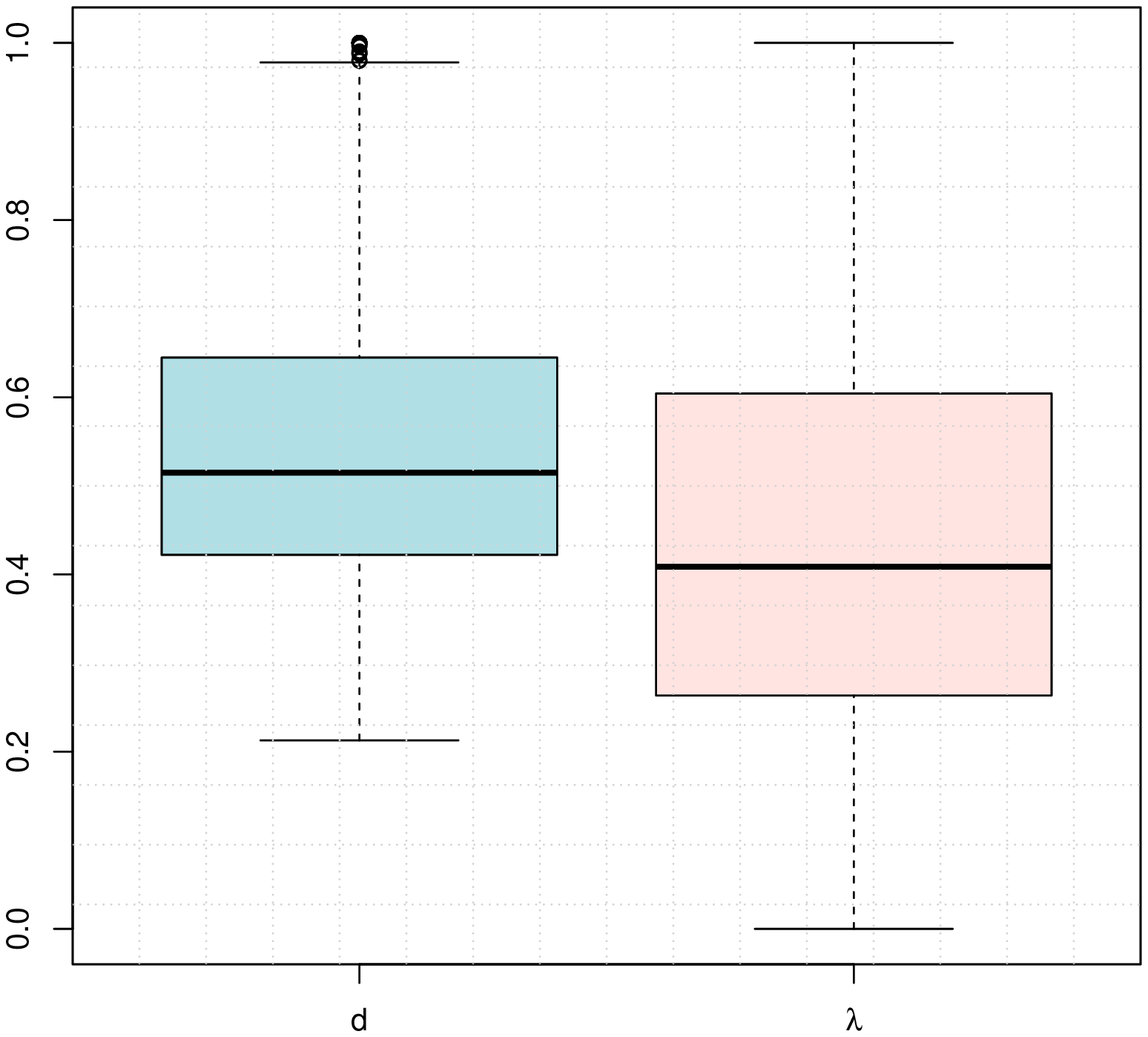} }}%
    \caption{Box plot of parameters using 1000 samples for $d = 0.4$ and $\lambda = 0.2$ (Fig \ref{fig:1a}) and for $d = 0.5$ and $\lambda = 0.4$ (Fig \ref{fig:1b}).}%
    \label{fig:1}%
\end{figure}

\subsection{Gegenbaur process}
The Gegenbauer polynomials are generalizations of the Legendre polynomials. For $|u|\le 1$ the Gegenbauer polynomials $C_n^d(u)$ are defined in terms of generating function as follows 
\begin{align}
(1-2uZ+Z^2)^{-d}=\sum_{n=0}^{\infty}C_n^d(u)Z^n,
\end{align}
where $d\neq0$, $|Z|<1$ and $C_n^d(u)$ is given by
\begin{equation}\label{Gegenbaur_coefficients}
    C_n^d(u)=\sum_{k=0}^{n/2}(-1)^k\dfrac{\Gamma(n-k+d)}{\Gamma(d)\Gamma(n+1)\Gamma(n-2k+1)}(2u)^{n-2k}.
\end{equation}
The concept of Gegenbauer process was developed by Andel (1986) \cite{Andel1986}. The Gegenbauer process is defined using the generating function of Gegenbauer polynomials. For a general linear process $\{X_t\}$ the Gegenbauer process is given by

\begin{align}\label{GP}
X_t=\sum_{n=0}^{\infty}C_n^d(u)\epsilon_{t-n},
\end{align}
where $\epsilon_t$ is white noise with mean $0$ and variance $\sigma^2$. In terms of backward shift operator $B$, \eqref{GP} can be written as
\begin{align*}
X_t=(1-2uB+B^2)^{-d}\epsilon_t.
\end{align*}

\begin{remark}
The Gegenbauer process is stationary for $|u|<1$ and $d<1/2$ or $|u| = \pm 1$ and $d<1/4$.
\end{remark}

\noindent The spectral density or the power spectrum of Gegenbauer process is given by \cite{Gray1989}
\begin{align*}
f_x(\omega)=\frac{\sigma^2}{\pi}\{4(\cos(\omega)-u)^2\}^{-d},
\end{align*}

\noindent and the autocovariance function of Gegenbauer process takes the following form
\begin{align*}
\gamma(h)=\frac{2^{1-2d}\sigma^2}{\pi}\sin^{-2d}(\omega_0)\sin(d\pi)\Gamma(1-2d)\cos(h\omega_0)\frac{\Gamma(h+2d)}{\Gamma(h+1)}.
\end{align*}

\noindent An extension to Gegenbauer process is known as Gegenbauer autoregressive integrated moving average (GARMA) process. GARMA process is a class of stationary long memory process which generalize the ARIMA and ARFIMA processes. For usefulness of generalized fractionally differenced Gegenbauer processes in time series modelling see \cite{Dissanayake2018} and for different estimation methods related to GARMA process see the recent article \cite{Hunt2022}. The process is defined as follows
\begin{align*}
    \Phi(B)(1-2uB+B^2)^d X_t=\epsilon_t \Theta(B),
\end{align*}
where $\epsilon_t$ is Gaussian white noise with variance $\sigma^2$, $B$ is the lag operator, $|u| \leq 1$ and $|d| <1/2$. Again, $\Phi(B)$ and $\Theta(B)$ are ARMA operators defined in \eqref{ARMA_operators}.

\section{Gegenbauer ARTFIMA process}\label{section3}
In this section, we introduce a new stochastic process namely Gegenbauer autoregressive tempered fractionally integrated moving average (GARTFIMA) process which is defined as follows
\begin{equation}\label{eq1_}
\Phi(B) (1-2ue^{-\lambda} B+e^{-2\lambda} B^2)^d X_t= \Theta(B) \epsilon_t,
\end{equation}
where $\epsilon_t$ is Gaussian white noise with variance $\sigma^2$, $B$ is the lag operator, $|u| < 1$, $\lambda\ge0$, $|d| <1/2$ and $\Phi(B)$, $\Theta(B)$ are ARMA operators defined in \eqref{ARMA_operators}. This process generalizes the ARIMA, ARFIMA, ARTFIMA and GARMA processes. In short this process is denoted by GARTFIMA$(p, d, \lambda, u, q)$.

\begin{proof}
Using \eqref{eq1_} it follows
\begin{align*}
X_t&= \Bigg(\frac{\Theta(B)}{\Phi(B)}\Bigg)(1-2ue^{-\lambda }B+e^{-2\lambda} B^2)^{-d} \epsilon_t,
\end{align*}
where we can write
\begin{align}\label{var}
(1-2u e^{-\lambda }z+e^{-2\lambda} z^2)^{-d}= \sum_{n=0}^{\infty}C_n^d(u)(e^{-\lambda}z)^n\;{\rm and}\; {\rm Var}(\epsilon_t) = \sigma^2.
\end{align}
For large $n$, the Gegenbauer polynomials $C_n^d(u)$ can be approximated as 
\begin{align*}
C_n^d(u)\sim \dfrac{\cos\left((n+d)\phi-d\pi/2\right)n^{d-1}}{\Gamma(d)\sin^d(\phi)}.
\end{align*}
Here $\phi=\cos^{-1}(u)$ and $\frac{\Theta(B)}{\Phi(B)}=\Psi(B)=\sum_{j=0}^{\infty}\psi_j B^j$. The variance of the process is given by
\begin{align*}
\mathrm{Var}(X_t)=\sigma^2\sum_{j=0}^{\infty}\psi_j^2\sum_{n=0}^{\infty}(C_n^d(u))^2 e^{-2\lambda n}
\end{align*}
and the variance will converge for $d<1/2$ and $\lambda\ge 0$ when $|u|<1$. To prove the invertibility condition we define the process \eqref{eq1_} as
\begin{align*}
\epsilon_t=\pi(B)X_t,
\end{align*}
where $\pi(z)=(1-2ue^{-\lambda}z+e^{-2\lambda}z^2)^d\frac{\Phi(z)}{\Theta(z)}$ and again using the same argument discussed above the $\pi(z)$ will converge for $d>-1/2$ and $\lambda\ge0$ when $|u|<1$. 
\end{proof}

\begin{theorem}
For a GARTFIMA$(p, d, \lambda, u, q)$ process defined in \eqref{eq1_}, the spectral density takes the following form
\begin{align*}
f_x(\omega)&=\dfrac{\sigma^2}{2\pi}\sum_{l=-q}^{q}\sum_{j=1}^{p}\psi(l)z^l\zeta_j\Bigg[\dfrac{\rho_j}{1-\rho_je^{\iota\omega}}-\dfrac{1}{1-\rho^{-1}e^{\iota\omega}}\Bigg](A-B \cos(\omega)+C \cos^2(\omega))^{-d},
\end{align*}
where $A=(1+4u^2e^{-2\lambda}-2e^{-2\lambda}+e^{-4\lambda}), B=4ue^{-\lambda}(1+e^{-2\lambda})$, $C= 4e^{-2\lambda},$
\[\psi(l)= \sum_{s=\min[0,l]}^{\max[q,q+l]}\theta_s\theta_{s-l}\]
and \[\zeta_j=\dfrac{\sigma^2}{2\pi}\Bigg[\rho_j\prod_{i=1}^{p}(1-\rho_i\rho_j)\prod(\rho_j-\rho_m)\Bigg]^{-1}.\]
\end{theorem}

\begin{proof}
\noindent Rewrite \eqref{eq1_} as follows
\begin{equation*}
X_t=\Psi(B)\epsilon_t,
\end{equation*}
where $\Psi(B)= \dfrac{\Theta(B)}{\Phi(B)} \Delta^{d,\lambda}$ and $\Delta^{d,\lambda}=(1-2ue^{-\lambda}B+B^2)^{-d}$. Then using the definition of spectral density of linear process, we have
\begin{align*}
f_x(\omega)=|\Psi(z)|^2f_\epsilon(\omega),
\end{align*}
where $z=e^{-\iota \omega}$ and $f_\epsilon(\omega)$ is spectral density of the innovation term. The spectral density of the innovation process $\epsilon_t$ is given by $\sigma^2/2\pi$, which implies
\begin{align}\label{SD}
f_x(\omega)=\dfrac{\sigma^2}{2\pi}|\Psi(z)|^2
=\dfrac{\sigma^2}{2\pi}\dfrac{|\Theta(z)|^2}{|\Phi(z)|^2}\left|(1-2ue^{-(\lambda+\iota\omega)}+e^{-2(\lambda+\iota\omega)})\right|^{-2d}.
\end{align}

\noindent Note that $\Phi(x)$ can be written as
\begin{align}\label{phi(x)}
\Phi(x)=\prod_{j=1}^{p}(1-\rho_jx),
\end{align}
where $\rho_1, \rho_2, \hdots, \rho_p$ are complex numbers such that $|\rho_j|<1$ for $j=1, 2, \hdots, p$. Using \eqref{phi(x)} in \eqref{SD}, it follows
\begin{align*}
f_x(\omega)&=\dfrac{\sigma^2|\Theta(e^{-\iota\omega})|^2 \left|(1-2ue^{-(\lambda+\iota\omega)}+e^{-2(\lambda+\iota\omega)})\right|^{-2d}}{2\pi \prod_{j=1}^{p}(1-\rho_jz)(1-\rho_jz^{-1})}.
\end{align*}
Using the results from \cite{Sowell1992}, the spectral density is
\begin{align*}
f_x(\omega)&=\dfrac{\sigma^2}{2\pi}\sum_{l=-q}^{q}\sum_{j=1}^{p}\psi(l)z^l\zeta_j\Bigg[\dfrac{\rho_j}{1-\rho_je^{\iota\omega}}-\dfrac{1}{1-\rho^{-1}e^{\iota\omega}}\Bigg]\left|(1-2ue^{-(\lambda+\iota\omega)}+e^{-2(\lambda+\iota\omega)})\right|^{-2d},
\end{align*}

\noindent where \[\psi(l)= \sum_{s=\min[0,l]}^{\max[q,q+l]}\theta_s\theta_{s-l}\] and \[\zeta_j=\dfrac{\sigma^2}{2\pi}\Bigg[\rho_j\prod_{i=1}^{p}(1-\rho_i\rho_j)\prod(\rho_j-\rho_m)\Bigg]^{-1}.\]

\noindent The spectral density takes the following form
\begin{align*}
f_x(\omega)&=\dfrac{\sigma^2}{2\pi}\sum_{l=-q}^{q}\sum_{j=1}^{p}\psi(l)z^l\zeta_j\Bigg[\dfrac{\rho_j}{1-\rho_je^{\iota\omega}}-\dfrac{1}{1-\rho^{-1}e^{\iota\omega}}\Bigg](A-B \cos(\omega)+C \cos^2(\omega))^d,
\end{align*}
where  $A=(1+4u^2e^{-2\lambda}-2e^{-2\lambda}+e^{-4\lambda}),\; B=4ue^{-\lambda}(1+e^{-2\lambda})$ and $C= 4e^{-2\lambda}.$
\end{proof}

\begin{theorem}
For the GARTFIMA$(p, d, \lambda, u, q)$ process defined in \eqref{eq1_}, the autocovariance function is given by
\begin{align*}
\gamma(h)=\sum_{l=-q}^{q}\sum_{j=1}^{p}\psi(l)\zeta_j\Bigg[\rho^{2p} \sum_{m=0}^{\infty}\rho^m \gamma_w(h-m) d\omega+\sum_{n=1}^{\infty}\rho^n \gamma_w(h+n) \Bigg],
\end{align*}
where $\gamma_w(h-m)=\sigma^2\sum_{n=0}^{\infty}C_n^d(u)C_{n+h}^d(t) e^{-(2n+(h-m)\lambda)}$ and $C_n^d(u)$ are defined in \eqref{Gegenbaur_coefficients}.
\end{theorem}
\begin{proof}

\noindent The autocovariance function can be calculated by taking the inverse Fourier transform of spectral density of the process $X_t$ using the following form

\begin{align*}
\gamma(h)&= \int_{-\pi}^{\pi}f_x(\omega) e^{\iota\omega h} d\omega\\
&=\dfrac{\sigma^2}{2\pi}\int_{-\pi}^{\pi}\Bigg[\sum_{l=-q}^{q}\sum_{j=1}^{p}\psi(l)\zeta_j\Bigg(\dfrac{\rho_j}{1-\rho_je^{\iota\omega}}-\dfrac{1}{1-\rho^{-1}e^{\iota\omega}}\Bigg)\left|(1-2ue^{-(\lambda+\iota\omega)}+e^{-2(\lambda+\iota\omega)})\right|^{-2d}z^{p+l}\Bigg] e^{\iota\omega h} d\omega.
\end{align*}

\noindent We can write another form using following expansions
\begin{align*}
\dfrac{\rho}{1-\rho e^{\iota\omega}}&=\rho^{2p}\sum_{m=0}^{\infty}(\rho e^{-\iota\omega})^m,\\
\dfrac{-1}{1-\rho^{-1}e^{\iota\omega}}&=-1+\sum_{n=0}^{\infty}(\rho e^{-\iota\omega})^n=\sum_{n=1}^{\infty}(\rho e^{\iota\omega})^n.
\end{align*}

\noindent The spectral density of GARTFIMA$(0, d, \lambda, u, 0)$ process $W_t=(1-2ue^{-\lambda}B+e^{-2\lambda}B^2)^{-d}\epsilon_t$ is given by

\begin{align*}
f_w(\omega)=\dfrac{\sigma^2}{2\pi}\left|(1-2ue^{-(\lambda+\iota\omega)}+e^{-2(\lambda+\iota\omega)})\right|^{-2d}.
\end{align*}

\noindent We can write 
\begin{align}\label{acov1}
\gamma(h)&=\dfrac{\sigma^2}{2\pi}\int_{-\pi}^{\pi}\Bigg[\sum_{l=-q}^{q}\sum_{j=1}^{p}\psi(l)\zeta_j\Bigg(\rho^{2p}\sum_{m=0}^{\infty}(\rho e^{-\iota\omega})^m+\sum_{n=1}^{\infty}(\rho e^{\iota\omega})^n\Bigg)\left|(1-2ue^{-(\lambda+\iota\omega)}+e^{-2(\lambda+\iota\omega)})\right|^{-2d}z^{p+l}\Bigg] e^{\iota\omega h} d\omega\nonumber\\
&=\dfrac{\sigma^2}{2\pi}\sum_{l=-q}^{q}\sum_{j=1}^{p}\psi(l)\zeta_j\Bigg[ \int_{-\pi}^{\pi} \rho^{2p}\sum_{m=0}^{\infty}(\rho e^{-\iota\omega})^m\left|(1-2ue^{-(\lambda+\iota\omega)}+e^{-2(\lambda+\iota\omega)})\right|^{-2d}e^{\iota\omega h} d\omega\nonumber\\ 
&\hspace{1cm}+\int_{-\pi}^{\pi}\sum_{n=1}^{\infty}(\rho e^{\iota\omega})^n\left|(1-2ue^{-(\lambda+\iota\omega)}+e^{-2(\lambda+\iota\omega)})\right|^{-2d} e^{\iota\omega h}d\omega\Bigg]\nonumber\\
&=\sum_{l=-q}^{q}\sum_{j=1}^{p}\psi(l)\zeta_j\Bigg[\rho^{2p} \sum_{m=0}^{\infty}\rho^m \int_{-\pi}^{\pi} \dfrac{2\pi}{\sigma^2} f_w(\omega)e^{\iota(h-m)} d\omega+\sum_{n=1}^{\infty}\rho^n\int_{-\pi}^{\pi} \dfrac{2\pi}{\sigma^2} f_w(\omega)e^{\iota(h+n)}d\omega \Bigg]\nonumber\\
&=\sum_{l=-q}^{q}\sum_{j=1}^{p}\psi(l)\zeta_j\Bigg[\rho^{2p} \sum_{m=0}^{\infty}\rho^m \gamma_w(h-m) d\omega+\sum_{n=1}^{\infty}\rho^n \gamma_w(h+n) \Bigg].
\end{align}

\noindent The autocovariance function for GARTFIMA $(0, d, \lambda, u, 0)$ process $W_t=(1-2ue^{-\lambda}B+e^{-2\lambda}B^2)\epsilon_t$ is 
\begin{align}\label{eq1}
\gamma_w(k)&=\int_{-\pi}^{\pi} \cos(k\omega)f_w(\omega) d\omega
=\dfrac{\sigma^2}{2\pi}\int_{-\pi}^{\pi} \cos(k\omega)\left|(1-2ue^{-(\lambda+\iota\omega)}+e^{-2(\lambda+\iota\omega)})\right|^{-2d} d\omega\nonumber\\
&=\dfrac{\sigma^2}{2\pi}\int_{-\pi}^{\pi}\cos(h\omega)\left|\sum_{n=0}^{\infty}C_n^d(u)(e^{-\lambda-\iota\omega})^n\right|^2 d\omega.
\end{align}
For ease of notation, let $C_n^d(u)=a_n$. Then

\begin{align}\label{eq2}
\left|\sum_{n=0}^{\infty}C_n^d(u)(e^{-\lambda-\iota\omega})^n\right|^2&=(a_0+a_1e^{-\lambda-\iota\omega}+a_2e^{-2\lambda-2\iota\omega}+\hdots)(a_0+a_1e^{-\lambda+\iota\omega}+a_2e^{-2\lambda+2\iota\omega}+\hdots)\nonumber\\
&=\sum_{n=0}^{\infty}a_n^2e^{-2n\lambda}+2\sum_{n=0}^{\infty}a_na_{n+1}e^{-(2n+1)\lambda}\cos(\omega)+2\sum_{n=0}^{\infty}a_n a_{n+2}e^{-(2n+2)\lambda}\cos(2\omega)+\hdots\nonumber\\
&=\sum_{n=0}^{\infty}a_n^2 e^{-2n\lambda}+2\sum_{r=1}^{\infty}\sum_{n=0}^{\infty}a_n a_{n+r}e^{-(2n+r)\lambda}\cos(\omega r).
\end{align}
\noindent Using \eqref{eq2} and \eqref{eq1}, it follows

\begin{align}\label{acov}
\gamma_w(k)&=\dfrac{\sigma^2}{\pi}\int_{-\pi}^{\pi}\cos(k\omega)\sum_{r=1}^{\infty}\sum_{n=0}^{\infty}a_n a_{n+r}e^{-(2n+r)\lambda}\cos(\omega r) d\omega\nonumber\\ 
&=\dfrac{\sigma^2}{\pi}\sum_{r=1}^{\infty}\sum_{n=0}^{\infty}a_n a_{n+r}e^{-(2n+r)\lambda}\int_{-\pi}^{\pi}\cos(k\omega)\cos(\omega r) d\omega
=\sigma^2\sum_{n=0}^{\infty}a_n a_{n+h}e^{-(2n+k)\lambda}.
\end{align}
Finally, taking $k=h-m$ and $k=h+n$ in \eqref{acov} and putting the values of $\gamma_w(h-m)$ and $\gamma_w(h+n)$ in \eqref{acov1}, one gets the desired result.
\end{proof}

\begin{proposition}
For $u=1$, the GARTFIMA$(p, d, \lambda, u, q)$ process reduced to ARTFIMA$(p, 2d, \lambda, q)$ process.
\end{proposition}
\begin{proof}
For $u=1$ the process defined in \eqref{eq1_} can be rewritten as
\begin{align*}
    \Phi(B)(1-e^{-\lambda}B)^{2d}X_t=\Theta(B)\epsilon_t.
\end{align*}
For $u=1$ the autocovariance function $\gamma_w(h)$ for GARTFIMA$(0, d, \lambda, u, 0)$ takes the following form
\begin{align*}
    \gamma_w(k)=\sigma^2\sum_{n=0}^{\infty}C_n^d(1)C_{n+h}^d(1)e^{-(2n+k)\lambda},
\end{align*}
where $C_n^d(1)=\binom{2d+n-1}{n} \text{ and } C_n^d(1)=\binom{2d+n+k-1}{n+k}.$ Now

\begin{align}\label{acov2}
    \gamma_w(k)&=\sigma^2(C_0^d(1)C_k^d(1)e^{-k\lambda}+C_1^d(1)C_{k+1}^d(1)e^{-(k+1)\lambda}+C_2^d(1)C_{k+2}^d(1)e^{-(k+2)\lambda}+\hdots)\nonumber\\
    &=\sigma^2\Bigg[\binom{2d+k-1}{k}e^{-k\lambda}+\binom{2d}{1}\binom{2d+k}{k+1}e^{-(k+2)\lambda}+\binom{2d+1}{2}\binom{2d+k+1}{k+2}e^{-(k+4)}\lambda+\hdots\Bigg]\nonumber\\
    &=\sigma^2\bigg[\frac{(2d+k-1)!e^{-k\lambda}}{(2d-1)!k!}+\frac{2d (2d+k)!e^{-(k+2)\lambda}}{(k+1)!(2d-1)!}+\frac{(2d+1)!(2d+k+1)!e^{-(k+4)\lambda}}{2!(2d-1)!(k+2)!(2d-1)!}+\hdots\Bigg]\nonumber\\
    &=\sigma^2e^{-k\lambda}\frac{2d(2d+1)\hdots(2d+k-1)}{k!}\Bigg[1+\frac{2d(2d+k)e^{-2\lambda}}{(k+1)}+\frac{2d(2d+1)(k+2d)(k+2d+1)e^{-4\lambda}}{(k+1)(k+2)2!}+\hdots\Bigg]\nonumber\\&=\frac{\sigma^2\Gamma(2d+k)}{\Gamma(2d)\Gamma(k+1)} {}_2F_1(2d,k+2d;k+1;e^{-2\lambda}),
\end{align}

\noindent which is the autocovariance function for ARTFIMA$(0, 2d, \lambda, 0)$ process defined in \cite{Sabzikar2019}. According to Sabzikar et al. the autocovariance function of ARTFIMA$(p,2d,\lambda,q)$ process is given by
\begin{align*}
\gamma(h)=\frac{\sigma^2}{2\pi}\sum_{l=-q}^{q}\sum_{j=1}^{p}\psi(l)\zeta_j\Bigg[\rho^{2p}\sum_{m=0}^{\infty}\rho^m\frac{2\pi}{\sigma^2}\gamma_w(h-m)+\sum_{n=1}^{\infty}\rho^n\frac{2\pi}{\sigma^2}\gamma_w(h+n)\Bigg]
\end{align*}
where $\gamma_w(k)$ for $k=h-m$ and $k=h+n$ is defined in equation \ref{acov2}. Using the $\gamma_w(k)$ from equation \ref{acov2} the autocovariance  of ARTFIMA$(p,2d,\lambda,q)$ takes a similar form to GARTFIMA$(p,d,\lambda,u,q)$ process for u=1. 

\end{proof}

\begin{proposition}
For $\lambda=0,$ the GARTFIMA$(p, d, \lambda, u, q)$ process reduced to GARMA$(p, d, u, q)$ process.
\end{proposition}
\begin{proof}
For $\lambda=0$, \eqref{eq1_} can be written as 
\begin{align*}
    \Phi(B)(1-2uB+B^2)^d X_t=\Theta(B)\epsilon_t,
\end{align*}
and for $\lambda=0$ in \eqref{acov} the autocovariance for GARMA$(0, d, u, 0)$ process $W_t$ takes the following form \cite{Woodward1998}
\begin{align*}
    \gamma_w(k)=\sigma^2\sum_{n=0}^{\infty}C_n^d(u)C_{n+k}^d(u).
\end{align*}
With the help of \eqref{acov1}, one can write the autocovariance function of GARMA$(p, d, u, q)$ in terms of autocovariance function of GARMA$(0, d, u, 0)$ denoted by $\gamma_w(h)$.
\end{proof}



\begin{theorem}
For a GARTFIMA$(p,d,\lambda,q)$ process $\sum_{h=0}^{\infty}\gamma(h)<\infty$ for $d<1/4$.
\end{theorem}
\begin{proof} We have,
\begin{align*}
\sum_{h=0}^{\infty}\gamma(h)&=\sigma^2\sum_{h=0}^{\infty}\sum_{n=0}^{\infty}C_n^d(u)e^{-n\lambda} C_{n+h}^d(u)e^{-(n+h)\lambda}.
\end{align*}
Hence
\begin{align*}
\sum_{h=0}^{\infty}|\gamma(h)|&\le \sum_{h=0}^{\infty} \sum_{n=0}^{\infty} |C_n^d(u)e^{-n\lambda}||C_{n+h}^d(u)e^{-(n+h)\lambda}|\\
&\le \sum_{h=0}^{\infty} \sum_{n=0}^{\infty}|C_n^d(u)||C_{n+h}^d(u)|\\&\le \sum_{h=0}^{\infty} \sum_{n=0}^{\infty}C_n^d(1)C_{n+h}^d(1)\\
&=\sum_{h=0}^{\infty} \sum_{n=0}^{\infty}\frac{\Gamma(2d+n)}{\Gamma(2d)\Gamma(n+1)}\frac{\Gamma(2d+n+h)}{\Gamma(2d)\Gamma(n+h+1)}
\end{align*}
Note that
\begin{equation}
\frac{\Gamma(2d+n)}{\Gamma(2d)\Gamma(n+1)}\frac{\Gamma(2d+n+h)}{\Gamma(2d)\Gamma(n+h+1)} \sim \left\{
	\begin{array}{ll}
		n^{4d-2}  & \mbox{as } n \rightarrow \infty,\;{\rm and }\; h\;{\rm finite}, \\
		h^{2d-1}  & \mbox{as } h \rightarrow \infty,\;{\rm and }\; n\;{\rm finite},\\
		2^{2d-1}n^{4d-2}  & \mbox{as } n \rightarrow \infty,\;{\rm and }\; h\rightarrow \infty.
	\end{array}
\right.
\end{equation}
Thus $\sum_{h=0}^{\infty}|\gamma(h)| < \infty$ if $d<1/4$, which completes the proof.
\end{proof}

\section{Parameter Estimation and Real-world Application}\label{section4}
In this section, the methodology for parameter estimation of GARTFIMA process is discussed. The parameter estimation of GARTFIMA process is done by adopting the non-linear least square (NLS) based approach discussed for ARTFIMA process in Section \ref{ARTFIMA}. Further, the parameter estimation is done using the Whittle likelihood method. These methods are adopted for estimating parameters $d, \lambda$ and $u$.  Firstly we will discuss the estimation using the NLS approach.\\

\noindent\textbf{Non-linear least square (NLS) approach:}\\
Similar to the procedure discussed for ARTFIMA process, let  $(1-2ue^{-\lambda} B+e^{-2\lambda} B^2)^d X_t=U_t$ and the spectral density of this process $U_t=\dfrac{\Theta(B)}{\Phi(B)}\epsilon_t$ is given by

\begin{align}\label{SD(GARTFIMA)}
f_u(\omega)=\dfrac{\sigma^2}{2\pi}\dfrac{|\Theta(e^{-\iota\omega})|^2}{|\Phi(e^{-\iota\omega})|^2}.
\end{align}

\noindent Substituting \eqref{SD(GARTFIMA)} in \eqref{SD}, the spectral density of $X_t$ can be written as

\begin{align}\label{SD1}
f_x(\omega)=f_u(\omega)\left|(1-2ue^{-(\lambda+\iota\omega)}+e^{-2(\lambda+\iota\omega)})\right|^{-2d}
= f_u(\omega)(A-B \cos(\omega)+C \cos^2(\omega))^{-d},
\end{align}
\noindent where $A=(1-4u^2e^{-2\lambda}-2e^{-2\lambda}+e^{-4\lambda}),\; B=2ue^{-\lambda}(1+e^{-2\lambda})$ and $C= 4e^{-\lambda}.$ Consider $\omega_j=\dfrac{2\pi j}{n}$, $j=0,1,\hdots,[n/2]$. Here $n$ is the sample size and $\omega_j$ is the set of harmonic frequencies. Taking natural logarithm on both sides of \eqref{SD1} with some manipulation leads to 
\begin{align*}
 \log\{f_x(\omega_j)\}&=\log\{f_u(0)\}+\log\{f_u(\omega_j)\}-d \log\{(A-B \cos(\omega_j)+C \cos^2(\omega_j))\}-\log\{f_u(0)\}\\
 &=\log\{f_u(0)\}-d \log\{(A-B \cos(\omega_j)+C \cos^2(\omega_j))\}+\log\Bigg\{\dfrac{f_u(\omega_j)}{f_u(0)}\Bigg\}.
\end{align*}
Adding $\log\{I(\omega_j)\}$ on both sides, it follows
\begin{align}\label{ref11}
 \log\{I(\omega_j)\}&=\log\{f_u(0)\}-d \log\{(A-B \cos(\omega_j)+C \cos^2(\omega_j))\}+\log\Bigg\{\dfrac{f_u(\omega_j)}{f_u(0)}\Bigg\}+\log\Bigg\{\dfrac{I(\omega_j)}{f(\omega_j)}\Bigg\},
\end{align}
where $I(\omega_j)$ is the periodogram or empirical spectral density of the process stated in \eqref{Iw}. Using \eqref{SD(GARTFIMA)}, $f_u(0)$ can be written as
\begin{align*}
    f_u(0)=\frac{\sigma^2}{2\pi}\frac{(1-\theta_1-\hdots-\theta_q)^2}{(1-\phi_1-\hdots-\phi_p)^2}.
\end{align*}

\noindent Now the estimation is done by minimizing the sum of squared errors in \eqref{ref11}, where the error is given by $\log(I(\omega_j)/f_x(\omega_j))$ and $\log\{f_u(0)\}$ is the intercept term for the equation. Also, choosing the $\omega_j$ near $0$, the term $\log\Bigg\{\dfrac{f_u(\omega_j)}{f_u(0)}\Bigg\}$ will be negligible. So the upper limit of $j$ should be chosen such that the $\omega_j$ is small or near zero. From \eqref{ref11}, we have
\begin{align}\label{ref22}
 \log\{I(\omega_j)\}&=\log\{f_u(0)\}-d \log\{(A-B \cos(\omega_j)+C \cos^2(\omega_j))\}+\log\Bigg\{\dfrac{I(\omega_j)}{f(\omega_j)}\Bigg\}.
\end{align}

\noindent One can estimate the parameters $d, \lambda$ and $u$ by using \eqref{ref11} or \eqref{ref22}. Generally, the estimates using \eqref{ref11} are better since it uses all the terms. In this article, we have done the parameter estimation by minimizing the squared error based on \eqref{ref11}, by using the R package {\it nloptr} which uses {\it nlxb} function for non-linear optimization. \\\

\noindent\textbf{Whittle likelihood estimation:}
 The Whittle likelihood estimation is a periodogram based technique. It employs spectral techniques to approximate the spatial likelihood which can be calculated using the Fast Fourier transform or spectral density of the time series $\{X_t\}$.  Consider the set of harmonic frequencies $\omega_j$, $j=0,1,\hdots, n/2$, the empirical spectral density is

\begin{align*}
I_x(\omega_j)=\dfrac{1}{2\pi}\Bigg\{R(0)+\sum_{s=1}^{n-1}R(s)\cos(s\omega_j)\Bigg\},
\end{align*}
where $\omega_j=2\pi j/n$, $j=0,1,\hdots, n/2$. The spectral density of the GARTFIMA$(0, d, \lambda, u, 0)$ process is
\begin{align}\label{SD2}
f_x(\omega)=\dfrac{\sigma^2}{2\pi}|\Psi(z)|^2
=\dfrac{\sigma^2}{2\pi}\dfrac{|\Theta(z)|^2}{|\Phi(z)|^2}(A-B \cos(\omega)+C \cos^2(\omega))^{-d}.
\end{align}

\noindent The whittle likelihood denoted by $l_w(\theta)$ is defined as

\begin{align*}
l_w(\theta)=\sum_{j=1}^{n}\frac{I_x(\omega_j)}{f_x(\omega)}+\log(f_x(\omega)),
\end{align*}
where $\theta$ is the unknown parameters vector given by $\theta = (d, \lambda, u)$. The estimates of the parameters $d, \lambda
$ and $u$ are obtained by minimizing the likelihood function $l_w(\theta)$ with respect to $\theta$.

\subsection{Simulation study}
To assess the performance of the introduced parameter estimation techniques, we use simulated data. The GARTFIMA process defined in \eqref{eq1_} can be written as
\begin{align}\label{new_eq}
\Phi(B)X_t=\Theta(B)(1-2ue^{-\lambda}B+e^{-2\lambda}B^2)^{-d}\epsilon_t.
\end{align}
To simulate a series from GARTFIMA process first we simulate an innovation series $\epsilon_t\sim {N}(\mu,\sigma^2)$. Then we generate another series $\eta_t=(1-2ue^{-\lambda}B+e^{-2\lambda}B^2)^{-d}\epsilon_t$ by taking the binomial expansion of $(1-2ue^{-\lambda}B+e^{-2\lambda}B^2)^{-d}$ and ignoring higher order terms. The process \eqref{new_eq} takes the following form
\begin{align*}
\Phi(B)X_t=\Theta(B)\eta_t,
\end{align*}
which is an ARMA$(p,q)$ process. Then using the package arima in R the GARTFIMA process can be simulated. We generate two series using different parameter combinations and evaluate the model performance using the parameter estimation based on non-linear regression. Assuming the innovation $\epsilon_t\sim \mathcal{N}(0,2)$ and taking $d=0.4$, $\lambda = 0.2$ and $u=0.1$ the series $\eta_t$ is simulated and then taking lags $p=1$ and $q=0$ we generate a synthetic GARTFIMA$(1, 0.4, 0.2, 0.1, 0)$ series. We also consider a synthetic GARTFIMA$(1, 0.5, 0.3, 0.2, 0)$ series. The actual and estimated parameters are shown in Table \ref{table:2}.

\begin{table}[ht!]
\begin{center}
\begin{tabular}{||c||c||c|}
\hline 
 & Actual & Estimated \\ 
\hline 
Case 1 & $d = 0.4$, $\lambda = 0.2$, $u = 0.1$& $\hat{d} = 0.41$, $\hat{\lambda} = 0.18$, $\hat{u} = 0.11$\\ 
\hline 
Case 2&  $d = 0.5$, $\lambda = 0.3$, $u = 0.2$ & $\hat{d} = 0.51$, $\hat{\lambda} = 0.30$, $\hat{u} = 0.19$\\ 
\hline 
\end{tabular}
\caption{Actual and estimated parameters values for different choices of parameters using NLS approach.}\label{table:2}
\end{center}
\end{table}

\noindent Moreover, to measure the effectiveness of the NLS technique based on empirical spectral density, box plots for different parameters are constructed. To construct the box plots a simulation of 1000 series assuming the parameters $d=0.4$, $\lambda=0.2$ and $u=0.1$ is done each with 1000 observations and the parameters $d$, $\lambda$ and $u$ are estimated using the NLS based estimation and the corresponding box plot for these estimated parameters from each simulation are shown in Fig \ref{fig:2} (left panel). Also, a simulation is performed for another combination of parameters that is $d=0.5$, $\lambda=0.3$ and $u=0.1$ and the box-plots for $1000$ simulations are given in Fig \ref{fig:2} (right panel).

\begin{figure}[ht!]
    \centering
    \subfloat[\label{fig:2a}]{{\includegraphics[width=8cm]{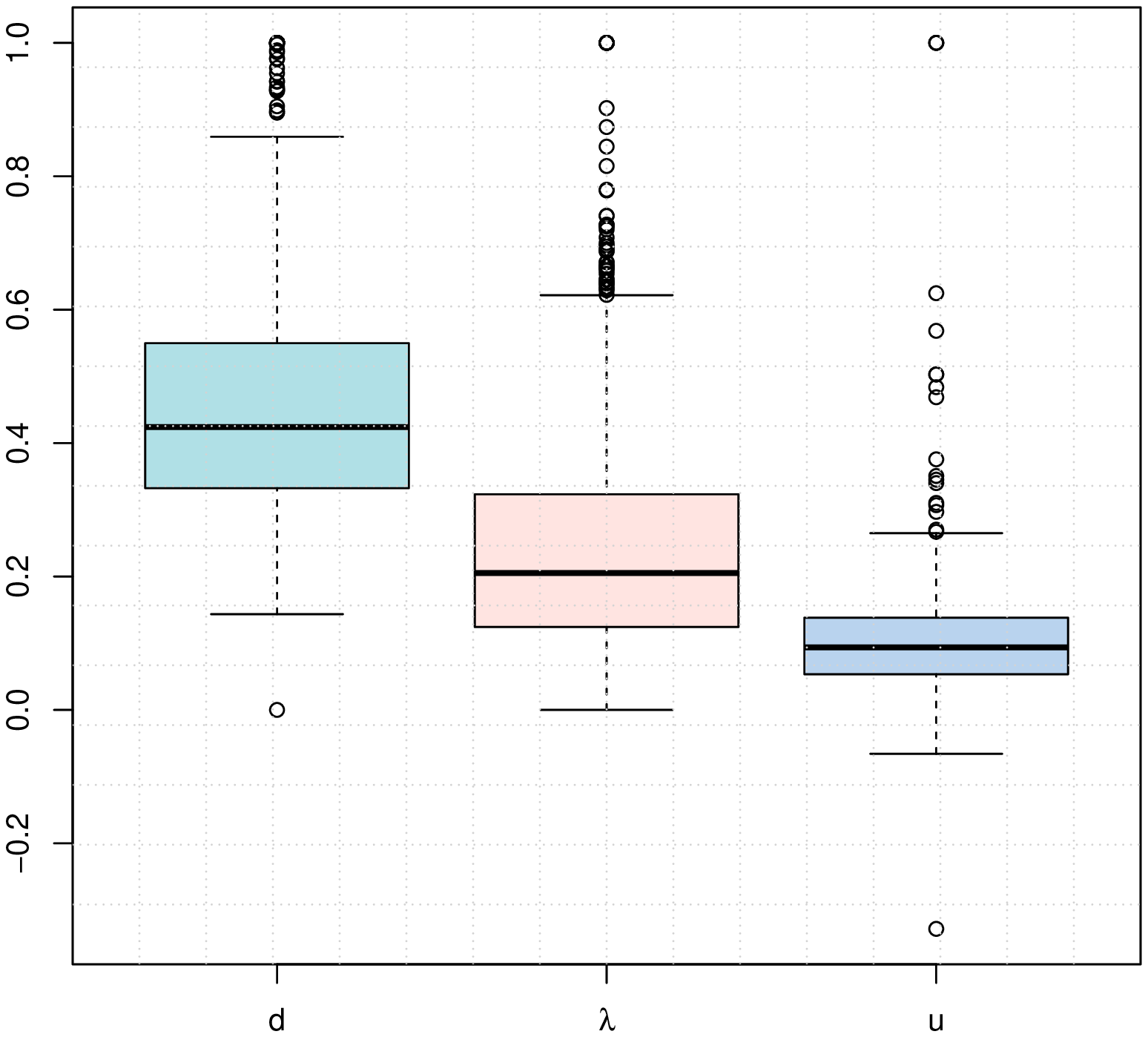} }}%
    \qquad
    \subfloat[\label{fig:2b}]{{\includegraphics[width=8cm]{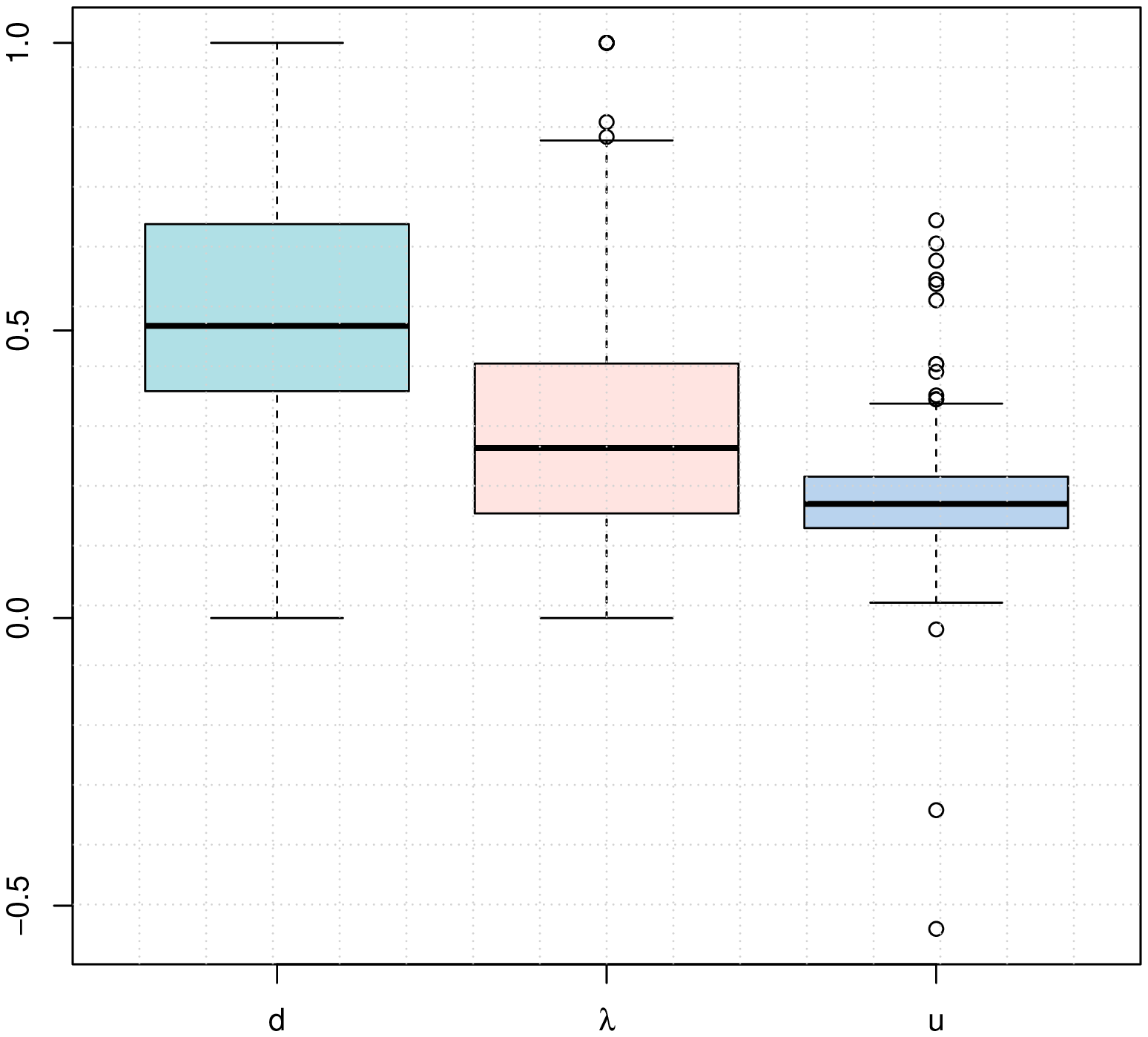} }}%
    \caption{Box plot of parameters using 1000 samples for $d = 0.4$, $\lambda = 0.2$ and $u=0.1$ (Fig \ref{fig:2a}) and for $d = 0.5$, $\lambda = 0.3$ and $u=0.2$ (Fig \ref{fig:2b}) based on NLS approach.}%
    \label{fig:2}%
\end{figure}

\noindent Next, we apply the Whittle likelihood technique on the same generated series with the same parameters combinations given above i.e. GARTFIMA$(1, 0.4, 0.2, 0.1, 0)$ and GARTFIMA$(1, 0.5, 0.3, 0.2, 0)$, and the corresponding estimates are summarized in table \ref{table:3}.

\begin{table}[ht!]
\begin{center}
\begin{tabular}{||c||c||c|}
\hline 
 & Actual & Estimated \\ 
\hline 
Case 1 & $d=0.4$, $\lambda=0.1$, $u=0.2$&$\hat{d}=0.37$, $\hat{\lambda}=0.08$, $\hat{u}=0.2$\\ 
\hline 
Case 2&  $d=0.5$, $\lambda=0.3$, $u=0.1$ & $\hat{d}=0.49$, $\hat{\lambda}=0.33$, $\hat{u}=0.08$\\ 
\hline 
\end{tabular}
\caption{Actual and estimated parameters values for different choices of parameters based on the Whittle likelihood.}\label{table:3}
\end{center}
\end{table}
\noindent Further, similar to the NLS estimation technique the parameters are estimated for 1000 simulations and each sample with 1000 observations using the Whittle likelihood  approach and box-plots are constructed for the same. The box-plots are shown in Fig \ref{fig:3}. 

\begin{figure}[H]
    \centering
    \subfloat[\label{fig:3a}]{{\includegraphics[width=8cm]{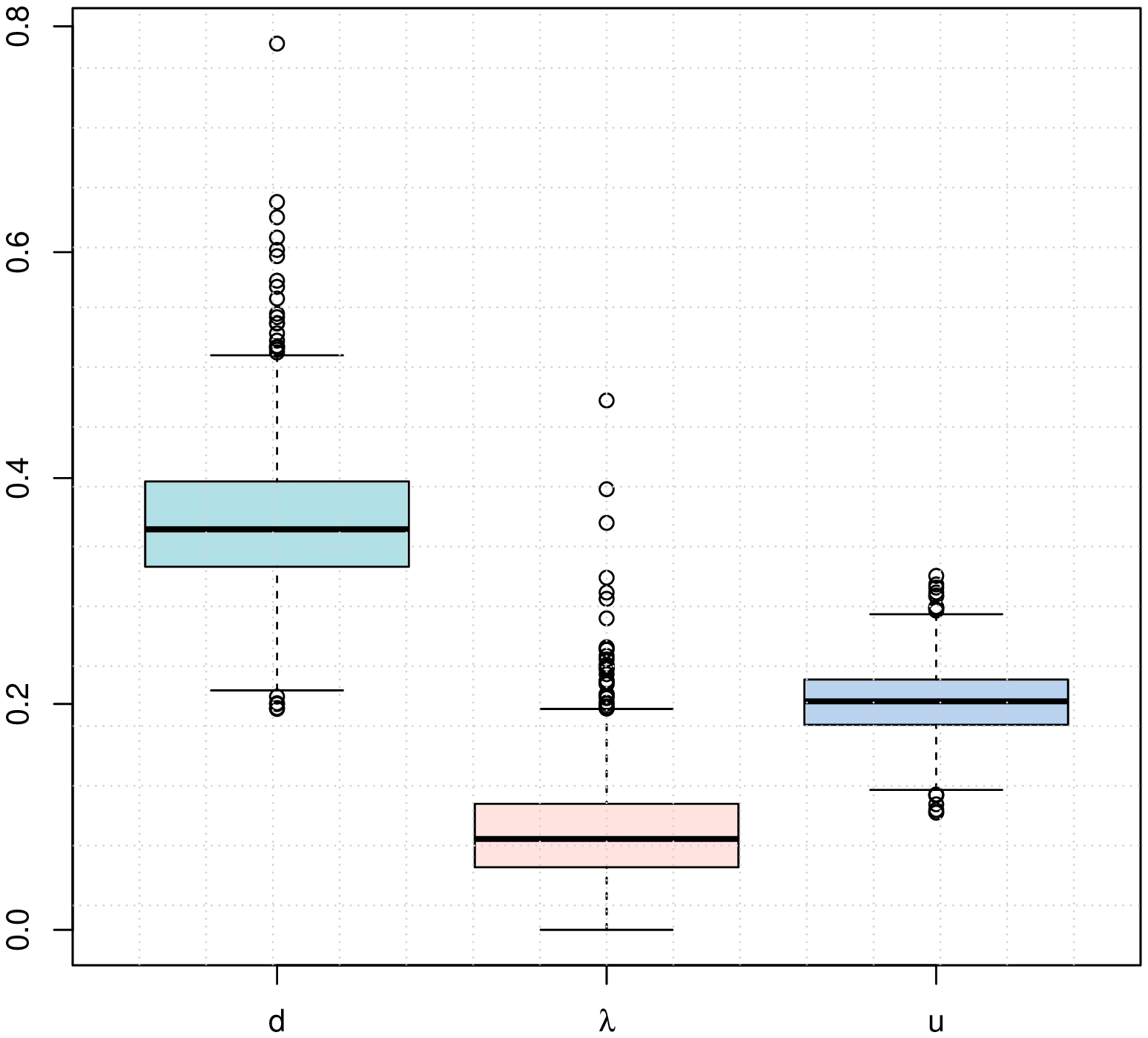} }}%
    \qquad
    \subfloat[\label{fig:3b}]{{\includegraphics[width=8cm]{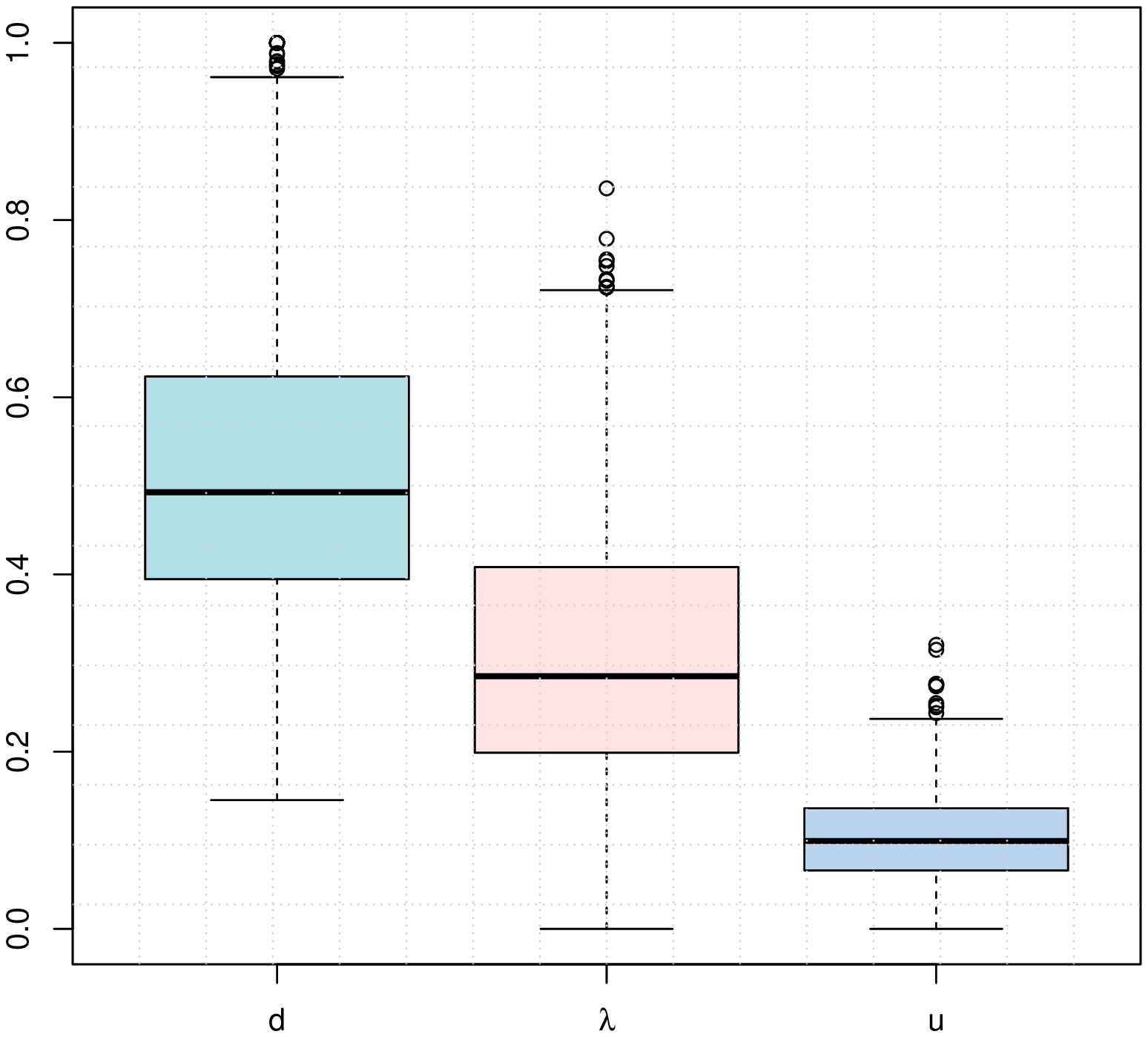} }}%
    \caption{Box plot of parameters using 1000 samples for $d = 0.4$, $\lambda = 0.1$ and $u=0.2$ (Fig \ref{fig:3a}) and for $d = 0.5$, $\lambda = 0.3$ and $u=0.2$ (Fig \ref{fig:3b}) based on the Whittle likelihood.}%
    \label{fig:3}%
\end{figure}

\subsection{Real data application}

To compare the performance of the introduced GARTFIMA process with existing time series models, two real-world data sets are considered. In this section, the comparison of the introduced model is done with existing time series models namely ARFIMA,  ARTFIMA and GARMA. We use two datasets for the comparison task and the first is the ``Nile annual minima" dataset\footnote{\url{https://rdrr.io/cran/FGN/man/NileMin.html}} which is a dataset of the annual minimum flow of the Nile river from 622 AD to 1284 AD. The ``Nile annual minima" dataset is already defined in R containing 663 observations. Another dataset is Spain's 10-year treasury bond daily percentage yield data\footnote{\url{https://www.investing.com/rates-bonds/spain-10-year-bond-yield-historical-data}} from July 2nd, 2012 to February 16th, 2017. The comparison study is given as follows.\\

\noindent\textbf{Nile river data}\\
\noindent The GARTFIMA model is applied to the ``Nile Annual Minima" dataset which is a dataset of annual minimum flow of the Nile river from The sample path of data is plotted in Fig \ref{fig:4a} also the ACF and PACF plots are given in Fig \ref{fig:4b} and \ref{fig:4c}. The ACF plot is significant for large lag values this indicates the presence of long memory property in the dataset. The AR and MA lags  are $p=3$ and $q=0$ respectively.

\begin{figure}[ht!]
    \centering
    \subfloat[Nile river annual minimum flow series \label{fig:4a}]{{\includegraphics[width=5cm]{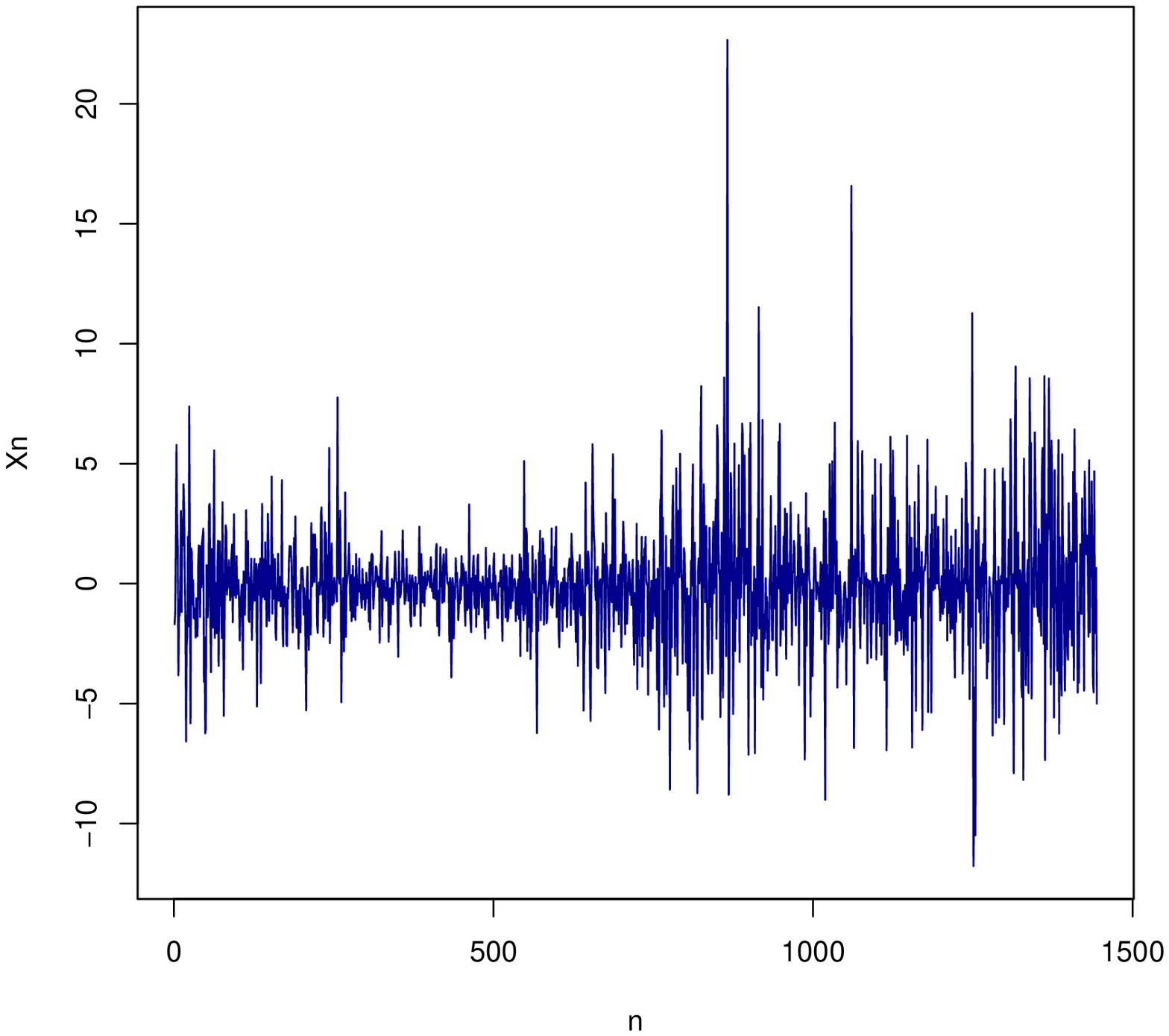} }}%
    \qquad
    \subfloat[ACF\label{fig:4b}]{{\includegraphics[width=5cm]{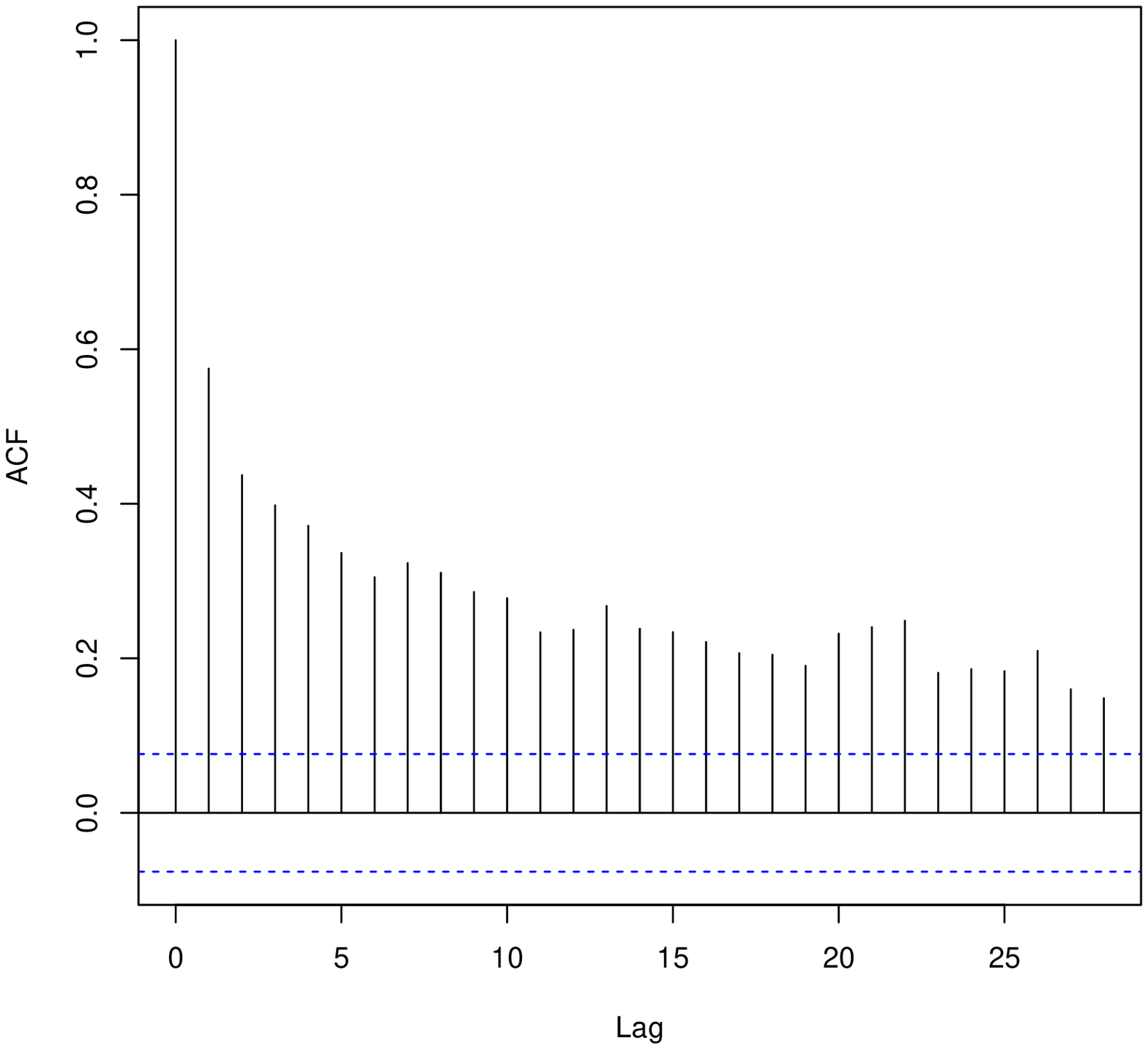} }}%
    \qquad
    \subfloat[PACF\label{fig:4c}]{{\includegraphics[width=5cm]{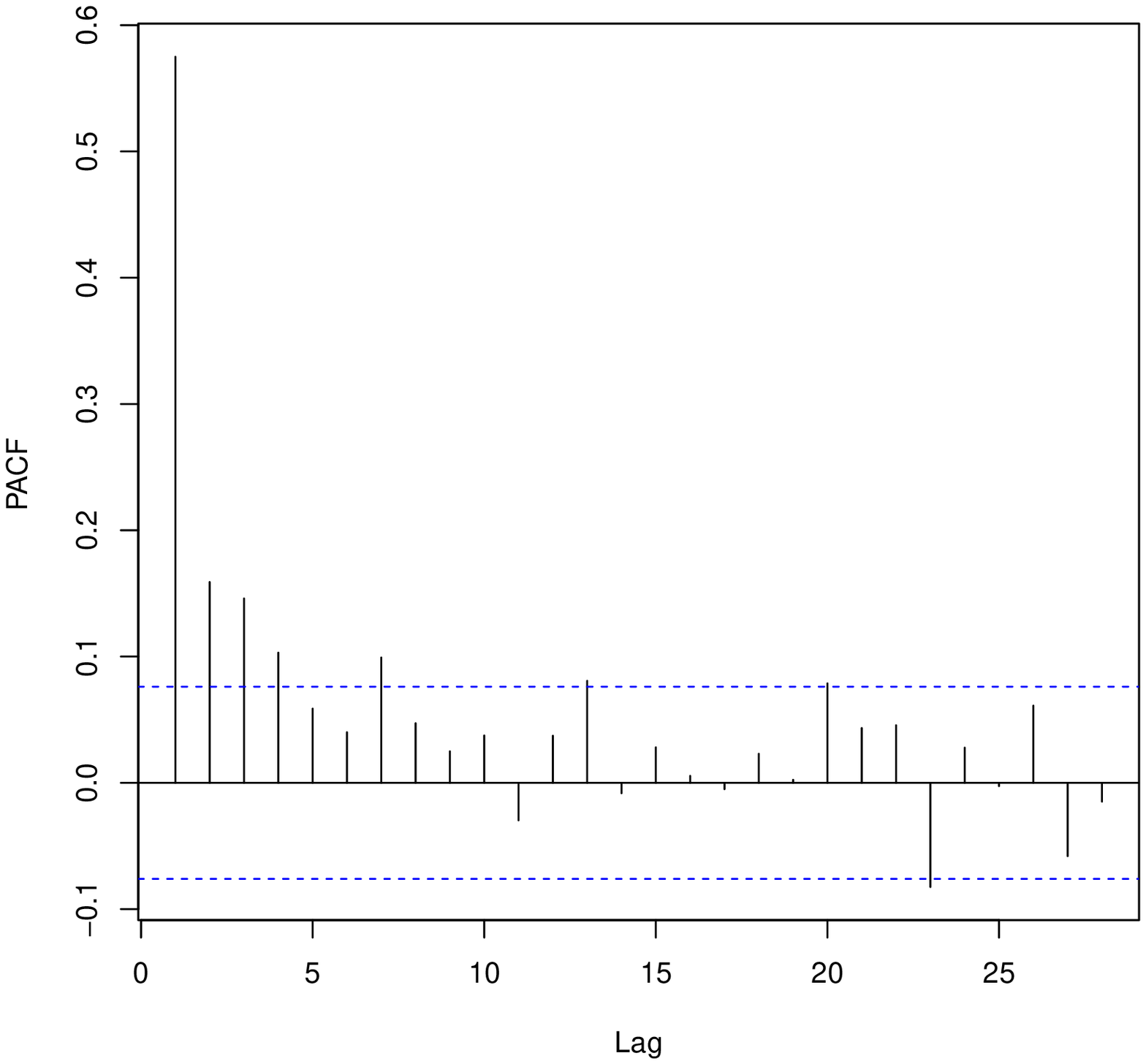} }}%
    \caption{Annual minimum flow series, ACF and PACF plot for Nile annual minima data from left to right respectively.}%
    \label{fig:4}%
\end{figure}

\noindent We apply the ARFIMA, ARTFIMA and GARTFIMA models to the Nile minima dataset and check the performance of each model on the dataset by looking at the root mean squared errors for each model. By taking almost $75\%$ of the data for the training set and $25\%$ of data in the test set we will train ARFIMA and ARTFIMA model using the R packages ``arfima" and ``artfima" respectively and check the performance on test set using ``rmse" values. Further, using the nonlinear least square estimation technique for GARTFIMA process the results are summarized in Table \ref{table:4}.  

\begin{table}[ht!]
\begin{center}
\begin{tabular}{||c||c||c|}
\hline 
 & Estimated parameters & RMSE value \\ 
\hline 
ARFIMA process & $\hat{d} = 0.39$& 0.72\\ 
\hline 
ARTFIMA process &  $\hat{d}=0.35$, $\hat{\lambda}=0.007$ &0.69 \\ 
\hline 
GARMA process &  $\hat{d}=0.2$, $\hat{u}=0.99$ &0.74 \\ 
\hline 
GARTFIMA process &  $\hat{d}=-0.16$, $\hat{\lambda}=0$, $\hat{u}=0.89$ & 0.68\\ 
\hline 
\end{tabular}
\caption{Model performance comparison where ARFIMA, ARTFIMA and GARMA are estimated using inbuilt R packages and GARTFIMA parameters are estimated using NLS estimation.}\label{table:4}
\end{center}
\end{table}

\noindent This indicates that the GARTFIMA process with NLS estimation approach fits the model equivalent to ARFIMA and ARTFIMA processes.  Further, the white noise can be plotted using actual and predicted values of the time series and to look for the normality of white noise series we plot the density of actual white noise with a synthetic white noise series. A synthetic random series of noise is generated from using the mean and variance of the actual series which are $0.063$ and $0.47$ respectively. The density plots for both the series are given in Fig. \ref{fig:5}, where the blue plot is for actual and black is for synthetic dataset.

\begin{figure}[H]
   \centering
   \includegraphics[width=0.6\textwidth, height = 0.4\textheight]{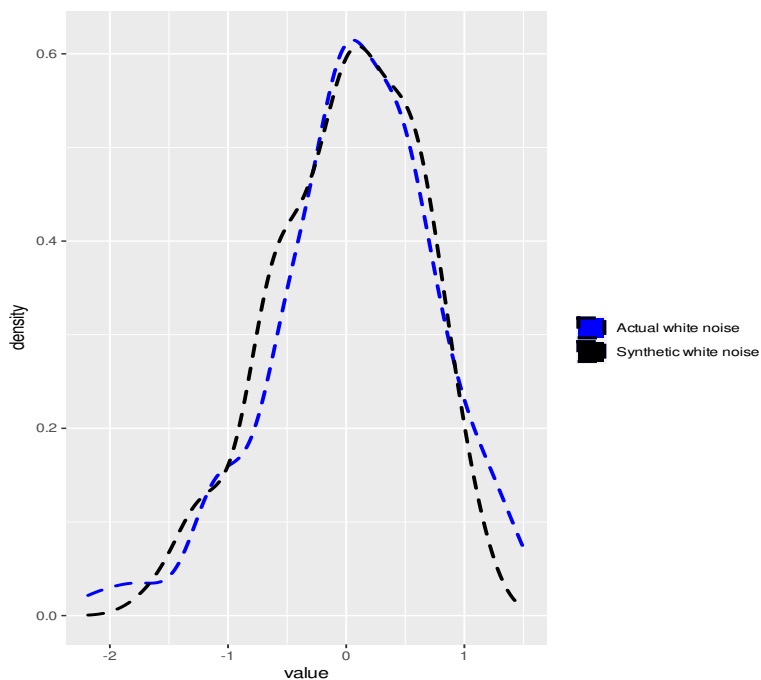}
   \caption{Density plot for actual and synthetic white noise series of the Nile annual minima dataset.}
   \label{fig:5}
\end{figure}

\noindent\textbf{Spain's 10-year treasury bond data}\\
The other dataset used for comparison is the 10-year bond yield of Spain which contains 1443 observations. The trajectory, ACF and PACF plots of the dataset is given in Fig \ref{fig:6a}, \ref{fig:6b} and \ref{fig:6c}. We get the AR and MA lags denoted by $p$ and $q$ and using Akaike information criterion (AIC) which comes out to be $p=3$ and $q=2$. 

\begin{figure}[ht!]
    \centering
    \subfloat[Spain's 10 year treasury yield series\label{fig:6a}]{{\includegraphics[width=5cm]{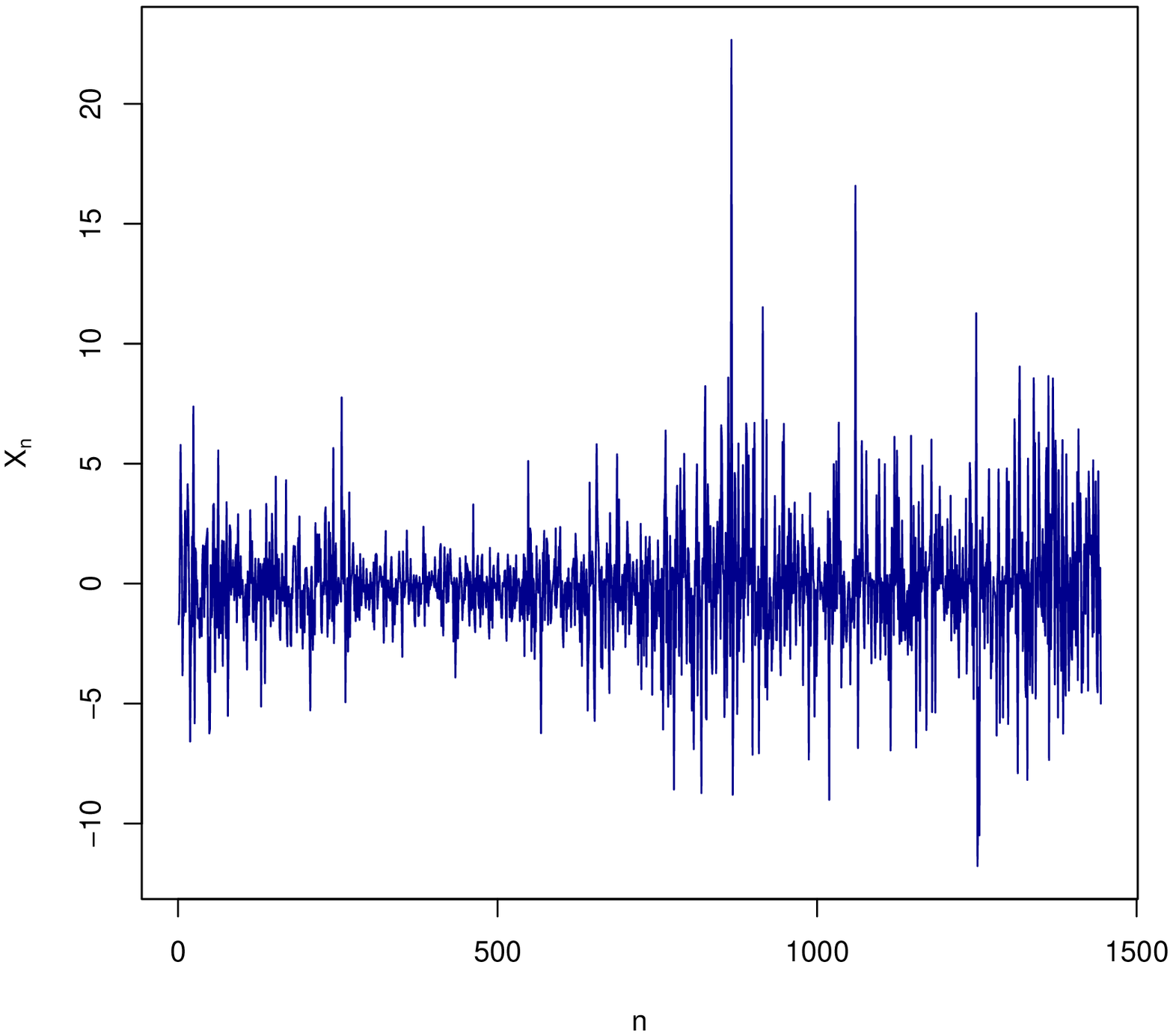} }}%
    \qquad
    \subfloat[ACF\label{fig:6b}]{{\includegraphics[width=5cm]{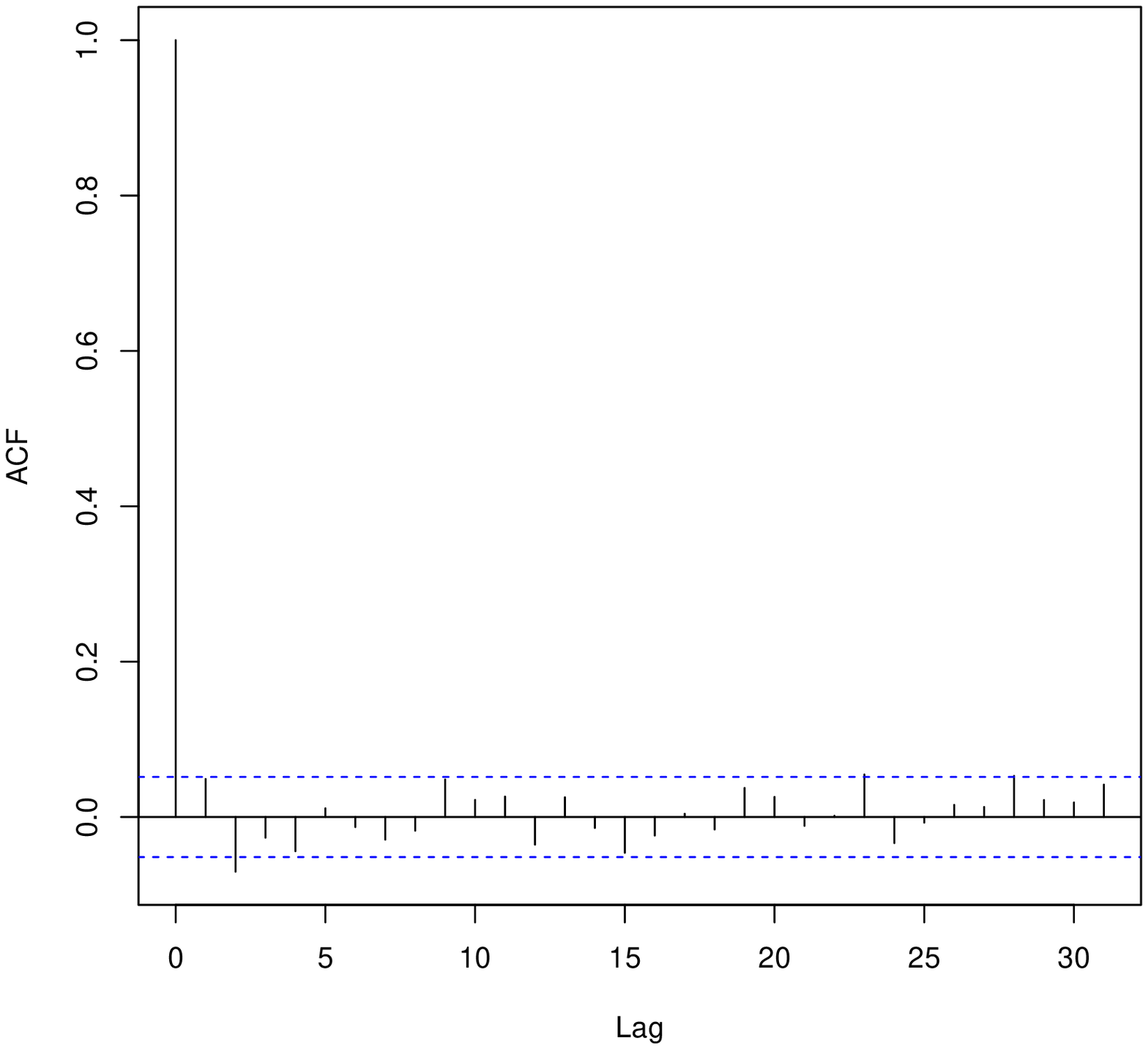} }}%
    \qquad
    \subfloat[PACF\label{fig:6c}]{{\includegraphics[width=5cm]{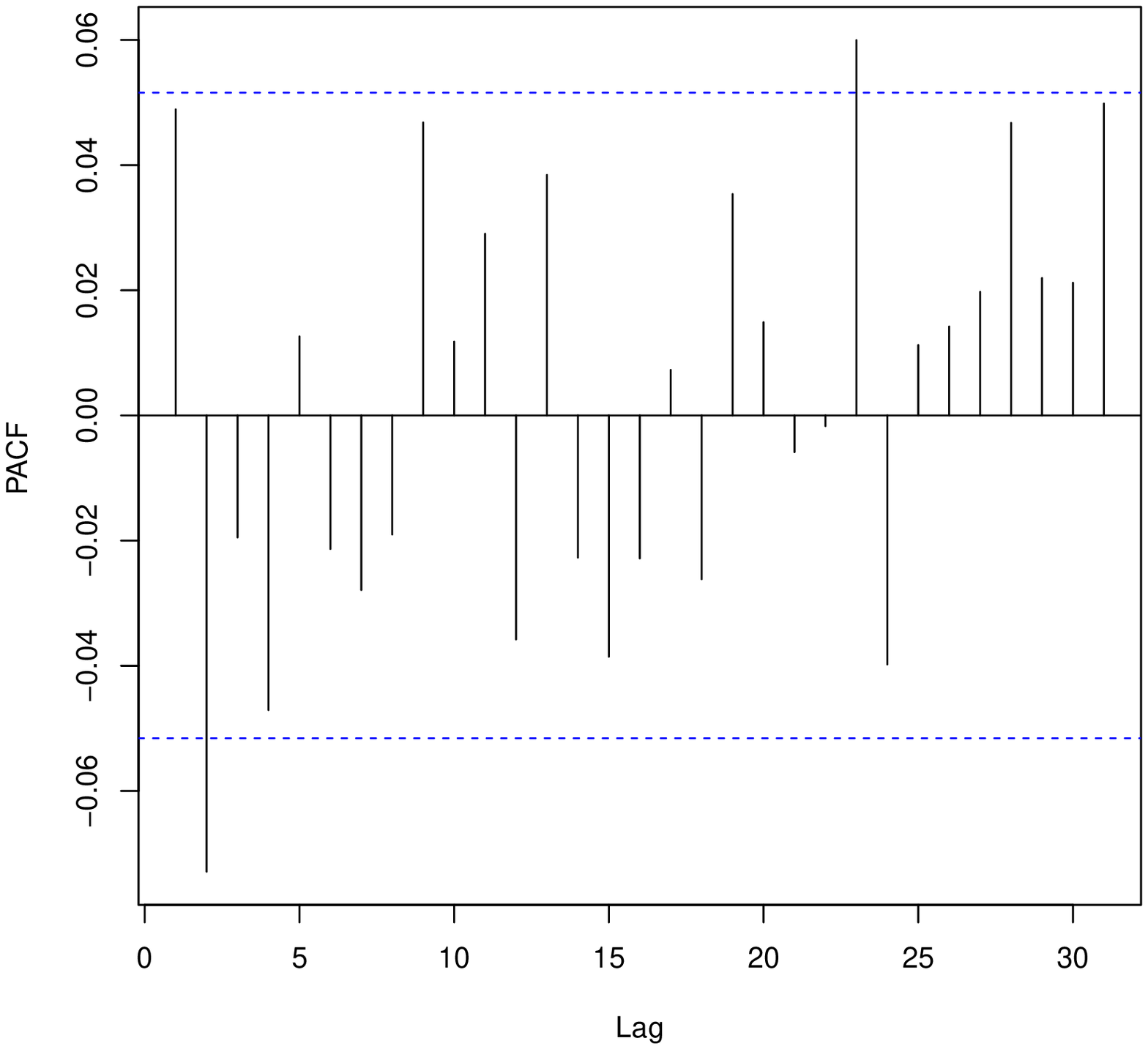} }}%
    \caption{Treasure yield series, ACF and PACF plots for Spain's 10-year bond data from left to right respectively.}%
    \label{fig:8}%
\end{figure}

\noindent Similar to the previous approach defined for Nile annual minima data, we use R to compare the performance of ARFIMA, ARTFIMA and GARTFIMA models. The results are summarized in the following table \ref{table:6}.
\begin{table}[ht!]
\begin{center}
\begin{tabular}{||c||c||c|}
\hline 
 & Estimated parameters & RMSE value \\ 
\hline 
ARFIMA process & $\hat{d} = -0.03$& 3.24\\ 
\hline 
ARTFIMA process &  $\hat{d}=-0.23$, $\hat{\lambda}=0.27$ &3.26 \\ 
\hline 
GARMA process & $\hat{d} = 0.07$, $\hat{u}=0.27$ & 3.24\\ 
\hline 
GARTFIMA process &  $\hat{d}=0.06$, $\hat{\lambda}=0.21$, $\hat{u}=0.41$ & 3.20\\ 
\hline 
\end{tabular}
\caption{Model performance comparison for GARTFIMA using the NLS estimation.}\label{table:6}
\end{center}
\end{table}
\noindent Table \ref{table:6} indicates that the GARTFIMA process using NLS technique for estimation performs slightly better than the ARFIMA and ARTFIMA processes. Moreover, a synthetic white noise series is generated using the mean and variance of actual white noise which are $0.48$ and $12.19$ respectively. Also, the density plots for both the actual and synthetic white noise series are given in Fig. \ref{fig:7} where the blue plot is for actual and black is for synthetic dataset.

\begin{figure}[H]
   \centering
   \includegraphics[width=0.6\textwidth, height = 0.4\textheight]{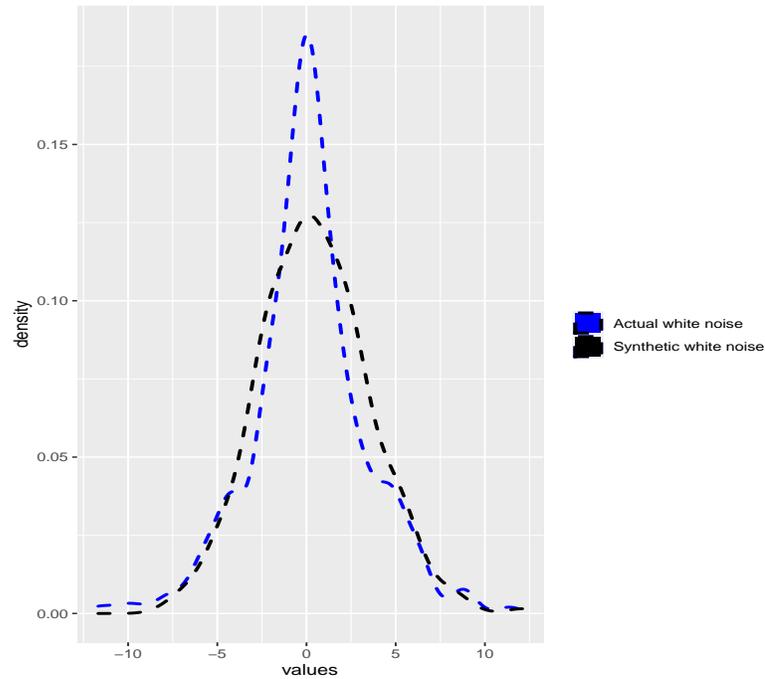}
   \caption{Density plot for actual and synthetic white noise series of the Spain's 10-year treasury bond dataset.}
   \label{fig:7}
\end{figure}

\section{Conclusions} In this article, we introduce GARTFIMA process which in particular cases includes ARIMA, ARFIMA, GARMA and ARTFIMA processes. The parameter estimation of the introduced model is done based on empirical spectral density. The estimation methods work satisfactorily as shown using simulated data and is depicted using box-plots. see Fig. \ref{fig:2} and Fig. \ref{fig:3}. The introduced model slightly better explains the Nile river data which is evident by comparing the RMSE with the other models. We believe that the model can be used in real-world data from diverse areas such as economics, finance, geophysics, ecology and others. It is well known that  generally financial time series have fat tails, demonstrate volatility clustering and possess nonlinear dependence. The introduced GARTFIMA model can be generalized to GARTFIMA-ARCH, GARTFIMA-GARCH, GARTFIMA model with heavy-tailed innovations or periodic GARTFIMA model. These generlized versions of GARTFIMA model will be helpful in better explaining real-world data.
\vspace{0.2cm}

\noindent {\bf Acknowledgements:} 
N.B. would like to thank Ministry of Education (MoE), India for supporting her PhD research.
\vone

\end{document}